\documentclass[graybox]{svmult}

\usepackage{mathptmx}       % selects Times Roman as basic font
\usepackage{helvet}         % selects Helvetica as sans-serif font
\usepackage{courier}        % selects Courier as typewriter font
\usepackage{type1cm}        % activate if the above 3 fonts are
\usepackage{makeidx}         % allows index generation
\usepackage{graphicx}        % standard LaTeX graphics tool
\usepackage{multicol}        % used for the two-column index
\usepackage[bottom]{footmisc}% places footnotes at page bottom

\usepackage{amssymb}
\usepackage{amsfonts}
\usepackage{psfrag}

\newcommand{\eps}{\varepsilon}
\newcommand{\dis}{\displaystyle}
\newcommand{\rens}{\mathbb{R}}
\newcommand{\R}{\mathbb{R}}
\newcommand{\EE}{\mathbb{E}}
\newcommand{\Var}{\mathbb{V}ar}

\newcommand{\dive}{\mbox{div }}

\newcommand{\text}[1]{\mbox{#1} }
\newcommand{\eqref}[1]{(\ref{#1})}
\newcommand{\tb}{\widetilde{b}}
\newcommand{\tsigma}{\widetilde{\sigma}}

\newcommand{\zbar}{\overline{z}}
\newcommand{\eeta}{\overline{\eta}}
\newcommand{\teta}{\widetilde{\eta}}

\newtheorem{assum}{\bf Assumption}

\makeindex             % used for the subject index
\begin{document}

\title*{Some remarks on free energy and coarse-graining}
\author{F. Legoll and T. Leli\`evre}
\institute{Fr\'ed\'eric Legoll \at Universit\'e Paris-Est, Institut
  Navier, LAMI, \'Ecole des Ponts ParisTech,
6 et 8 avenue Blaise Pascal, 77455 Marne-La-Vall\'ee Cedex 2, France and
INRIA Rocquencourt, MICMAC Team-Project, Domaine de Voluceau, B.P. 105,
78153 Le Chesnay Cedex, France.
\email{legoll@lami.enpc.fr}
\and Tony Leli\`evre \at Universit\'e Paris-Est, CERMICS, \'Ecole des
Ponts ParisTech,
6 et 8 avenue Blaise Pascal, 77455 Marne-La-Vall\'ee Cedex 2, France and
INRIA Rocquencourt, MICMAC Team-Project, Domaine de Voluceau, B.P. 105,
78153 Le Chesnay Cedex, France.
\email{lelievre@cermics.enpc.fr}}
\maketitle

\abstract{
We present recent results on coarse-graining techniques for
thermodynamic quantities (canonical averages) and dynamical quantities
(averages of path functionals over solutions of overdamped Langevin
equations). The question is how to obtain reduced models to compute such
quantities, in the specific case when the functional to be averaged only
depends on a few degrees of freedom. 
We mainly review, numerically illustrate and extend results
from~\cite{bllp,dyn_eff},
concerning the computation of the stress-strain relation for
one-dimensional chains of atoms, and the construction of an effective
dynamics for a scalar coarse-grained variable when the complete system
evolves according to the overdamped Langevin equation.
}

\section{Motivation}

In molecular simulation, two types of quantities are typically of
interest: averages with respect to the canonical ensemble ({\em thermodynamic
quantities}, such as stress, root-mean-square distance, \dots), and averages of
functionals over paths ({\em dynamic quantities}, like viscosity, diffusion
coefficients or rate constants). In both cases, the question of
coarse-graining is relevant, in the sense that the considered
functionals typically depend only on a few variables of the system
(collective variables, or reaction coordinates). Therefore, it is
essential to understand how to obtain coarse-grained models on these
variables.

\subsection{Coarse-graining of thermodynamic quantities}

Computing canonical averages is a standard task in molecular
dynamics. For a molecular system whose atom positions are described by a
vector $q \in \rens^n$, these quantities read 
$$
\int_{\R^n} \Phi(q) \, d\mu
$$
where $\Phi : \R^n \to \R$ is the observable of interest and $\mu$ is the
Boltzmann-Gibbs measure,
\begin{equation}
\label{eq:mu}
d\mu = Z^{-1} \exp(-\beta V(q)) \, dq,
\end{equation}
where $V$ is the potential energy of the system, $\beta$ is
proportional to the inverse of the system temperature, and $\dis{
Z=\int_{\R^n} \exp(-\beta V (q)) \, dq
}$
is a normalizing constant. Typically, $q$ represents the position of $N$
particles in dimension $d$, hence $q \in \rens^n$ with $n = dN$.

As mentioned above, observables of interest are often functions of only
part of the 
variable $q$. For example, $q$ denotes the positions of {\em all}
the atoms of a protein and of the solvent molecules around, and the
quantity of interest is only a particular angle between some atoms in
the protein, because this angle characterizes the conformation of the
protein (and thus the potential energy well in which the system is, is
completely determined by the knowledge of this quantity of
interest). Another example is the case when $q=(q^1,\dots,q^n)$ denotes
the positions of 
all the atoms of a one-dimensional chain, and quantities of interest are
only a function of the total length $q^n-q^1$ of the chain. 

We thus introduce the so-called {\em reaction coordinate} 
$$
\xi: \R^n \to \R,
$$
which contains all the information we are interested in.
Throughout this article, we assume that it is a smooth function such
that $|\nabla \xi|$ is bounded from below by a positive constant, so
that the configurational space can be foliated by isosurfaces associated
to $\xi$. A simple case that will be considered below is $\xi(q^1,
\ldots, q^n)=q^n$. 

To this function $\xi$ is naturally associated an effective
energy $A$, called the {\em free energy}, such that
$$
d(\xi \star \mu) = \exp(-\beta A(z))\, dz,
$$
where $\xi \star \mu$ denotes the image of the measure $\mu$ by $\xi$. In
other words, for any test function $\Phi : \R \to \R$,
\begin{equation}
\label{eq:stat}
\int_{\R^n} \Phi(\xi(q)) \ Z^{-1} \exp(-\beta V(q)) \, dq 
= 
\int_\R \Phi(z) \, \exp(-\beta A(z))\, dz.
\end{equation}
Expressions of $A$ and its derivative are given below (see Section
\ref{sec:notation}).

The interpretation of~\eqref{eq:stat} is that, when $Q$ is a random
variable distributed
according to the Boltzmann measure~\eqref{eq:mu}, then $\xi(Q)$ is
distributed according to the measure $\exp(-\beta A(z))\, dz$. 
Hence, the free energy $A$ is a relevant quantity for computing
thermodynamic quantities, namely canonical averages. 

In conclusion, the question of coarse-graining thermodynamic quantities
amounts to computing the free energy, and there are several efficient
methods to perform such calculations (see for
example~\cite{chipot-pohorille-07,lelievre-rousset-stoltz-10}). In the
sequel of this article, we address a particular case, motivated by
materials science, where the system under consideration is a
one-dimensional chain of atoms, and $\xi(q^1,\dots,q^n) = q^n-q^1$ is
the length of the chain (see Fig.~\ref{fig:chain} below). We are interested
in the free energy associated to this reaction coordinate, and its
behaviour when the number $n$ of particles goes to $+\infty$. 
Standard algorithms to compute the free energy then
become prohibitively expensive, as the dimension of the system becomes
larger and larger. Alternative strategies are needed, and 
we investigate analytical methods, based on large deviations
principles, in Section~\ref{sec:static}. 

\subsection{Coarse-graining of dynamical quantities}
\label{sec:intro_dyn}

The second topic of this contribution is related to the {\em
  dynamics} of the system, and how to coarse-grain it. In short, we will
show how to  
design a dynamics that approximates the path $t \mapsto
\xi(Q_t)$, where $\xi$ is the above reaction coordinate.

To make this question precise, we first have to {\em choose} the full
dynamics, which will be the reference one. 
In the following, we consider the overdamped Langevin dynamics on state
space~$\R^n$:
\begin{equation}
\label{eq:X}
dQ_t = - \nabla V(Q_t) \, dt + \sqrt{2 \beta^{-1}} \, d W_t,
\quad Q_{t=0} = Q_0,
\end{equation}
where $W_t$ is a standard $n$-dimensional Brownian motion.
Under suitable assumptions on~$V$, this dynamics is ergodic with respect
to the Boltzmann-Gibbs measure~\eqref{eq:mu} (see~\cite{comparisonNVT}
and references therein). Hence, for $\mu$-almost
all initial conditions $Q_0$,
\begin{equation}
\label{eq:ergo}
\lim_{T \to \infty} \frac{1}{T} \int_0^T \Phi(Q_t) \, dt = 
\int_{\R^n} \Phi(q) \, d\mu
\end{equation}
almost surely. In practice, this convergence is often very slow, due to some
metastabilities in the dynamics: $Q_t$ samples a given well of the
potential energy for a long time, before hoping to some other well of
$V$. 

An important dynamical quantity we will consider below is the
average residence time, that is the mean time that the system spends in a
given well, before hoping to another one, when it follows the dynamics
\eqref{eq:X}. Typically, the wells are fully described through $\xi$
($q$ is in a given well if and only if $\xi(q)$ is in a given interval),
so that these times can be obtained from the knowledge of the time
evolution of $\xi(Q_t)$, which is expensive to compute since it means 
simulating the full system.

In Section~\ref{sec:eff_dyn} below, we will first present a
one-dimensional dynamics of the form 
\begin{equation}
\label{eq:eff_dyn}
d\eeta_t = b(\eeta_t) \, dt + \sqrt{2 \beta^{-1}} \, \sigma(\eeta_t) \, d B_t,
\end{equation}
where $B_t$ is a standard one-dimensional Brownian motion and $b$ and
$\sigma$ are scalar functions, such that 
$\left( \eeta_t \right)_{0 \leq t \leq T}$
is a good approximation (in a sense to be made precise below) of
$\left( \xi(Q_t)\right)_{0 \leq t \leq T}$. Hence, the
dynamics~\eqref{eq:eff_dyn} can be  
thought of as a coarse-grained, or {\em effective}, dynamics for the quantity
of interest. A natural requirement is that \eqref{eq:eff_dyn} preserves
equilibrium quantities, {\em i.e.} it is ergodic with respect to
$\exp(-\beta A(z))\, dz$, the equilibrium measure of $\xi(Q_t)$ when
$Q_t$ satisfies~\eqref{eq:X}, but we 
typically ask for more than that. For example, we would like to be able to
recover residence times in the wells from~\eqref{eq:eff_dyn}, hence
bypassing the expensive simulation of $\xi(Q_t)$.

As a matter of fact, the coarse-grained dynamics
\begin{equation}
\label{eq:zbar}
d\zbar_t = - A'(\zbar_t) \, dt + \sqrt{2 \beta^{-1}} \, d B_t,
\end{equation}
is a one-dimensional dynamics that is ergodic with respect to
$\exp(-\beta A(z))\, dz$. It can thus be thought of as a natural
candidate for a dynamics approximating $\xi(Q_t)$, all the more so as
practitioners 
often look at the free energy profile ({\em i.e.} the function $z \mapsto
A(z)$) to get an idea of the dynamics of transition (typically the
transition time) between one region indexed by the reaction coordinate
(say for example $\left\{ q \in \R^n; \ \xi(q) \le z_0 \right\}$) and
another one (for example $\left\{ q \in \R^n; \ \xi(q) > z_0
\right\}$). If $\xi(Q_t)$ follows a dynamics which is close
to~\eqref{eq:zbar}, then the Transition State Theory says that residence 
times are a function of the free energy
barriers~\cite{kramers,kramers_review}, and then it makes sense to look
at the free energy to compute some dynamical properties.
It is thus often assumed that there is some dynamical information in 
the free energy $A$. 

In the sequel, we will compare the accuracy of both coarse-grained
dynamics, an effective dynamics of type~\eqref{eq:eff_dyn}
(namely dynamics~\eqref{eq:y} below) and the
dynamics~\eqref{eq:zbar} driven by the free energy. Their relation has
been investigated from an analytical viewpoint in~\cite[Section
2.3]{dyn_eff} (see also~\cite[Sec. 10 and Eq. (89)]{e-eve04}
and~\cite{maragliano06}).  

\subsection{Outline of the article}

In this contribution, we mainly review, numerically illustrate and extend
results from the two articles~\cite{bllp,dyn_eff}. The aim is to present
in a pedagogical and unified manner recent contributions on
coarse-graining procedures concerning: (i) a static case inspired by
material sciences, namely the computation of stress-strain relation for
one-dimensional chains of atoms, in the thermodynamic limit
(Section~\ref{sec:static}) and (ii) a dynamic case inspired by molecular
dynamics computations, namely the derivation of effective dynamics along
the reaction coordinate, for overdamped Langevin equations
(Section~\ref{sec:eff_dyn}). Compared to the original
articles~\cite{bllp,dyn_eff}, we propose some extensions of the
theoretical results (see {\em e.g.} Section~\ref{sec:NNN}), some simpler
proofs in more restricted settings (in Section~\ref{sec:pedago}) and new
numerical experiments (Sections~\ref{sec:NNN_num} and~\ref{sec:eff_dyn_num}). 

\subsection{Notation}
\label{sec:notation}

We gather here some useful notation and results. Let $\Sigma_z$ be
the submanifold of $\R^n$ of positions at a fixed value of the reaction coordinate:
\begin{equation}
\label{eq:mani}
\Sigma_z= \{ q \in \R^n; \, \xi(q) =  z \}.
\end{equation}
Let us introduce $\mu_{\Sigma_z}$, which is the probability measure
$\mu$ conditioned at a fixed value of the reaction coordinate:
\begin{equation}
\label{eq:mu_z}
d\mu_{\Sigma_z} = 
\frac{\exp(-\beta V) \, |\nabla \xi|^{-1} \, d\sigma_{\Sigma_z}}
{\dis{ \int_{\Sigma_z} \exp(-\beta V) \, |\nabla \xi|^{-1} \,
    d\sigma_{\Sigma_z}}}, 
\end{equation}
where the measure $\sigma_{\Sigma_z}$ is the Lebesgue measure
on~$\Sigma_z$ induced by the Lebesgue measure in the ambient Euclidean
space $\R^n \supset \Sigma_z$. 
By construction, if $Q$ is distributed according to the Gibbs
measure~\eqref{eq:mu}, then the law of $Q$ conditioned to a fixed value
$z$ of $\xi(Q)$ is $\mu_{\Sigma_z}$.
The measure $|\nabla \xi|^{-1} d\sigma_{\Sigma_z}$ is sometimes denoted
by $\delta_{\xi(q)-z}(dq)$ in the literature. 

We recall the following expressions for the free energy $A$ and its
derivative $A'$, also called the {\em mean force} (see~\cite{c_l_eve}):
\begin{eqnarray}
\label{eq:A}
A(z) &=& - \beta^{-1} \ln \left( \int_{\Sigma_z} Z^{-1} \, \exp(-\beta V)
  \, |\nabla \xi|^{-1} \, d\sigma_{\Sigma_z} \right),
\\
\label{eq:Aprime}
A'(z) &=& \int_{\Sigma_z} F \, d\mu_{\Sigma_z},
\end{eqnarray}
where $F$ is the so-called {\em local mean force}:
\begin{equation}
\label{eq:F}
F=\frac{\nabla V \cdot \nabla \xi}{|\nabla \xi|^2}  - \beta^{-1} \, \dive
\left( \frac{\nabla \xi}{|\nabla \xi|^2} \right).
\end{equation}
In the particular case when the reaction coordinate is just one of the
cartesian coordinate, say $\xi(q) = q^n$, then  
$$
A(z) = - \beta^{-1} \ln \left( \int_{\rens^{n-1}} Z^{-1} \, \exp(-\beta V(q^1,\dots,q^{n-1},z))
  \, dq^1 \dots dq^{n-1} \right)
$$
and the local mean force is just $F = \partial_{q^n} V$, so that
$$
A'(z)= \frac{
\int_{\rens^{n-1}} \partial_{q^n} V (q^1,\dots,q^{n-1},z) 
\exp(-\beta V(q^1,\dots,q^{n-1},z)) \, dq^1 \dots dq^{n-1}
}{
\int_{\rens^{n-1}} \exp(-\beta V(q^1,\dots,q^{n-1},z))
\, dq^1 \dots dq^{n-1}}.
$$

\section{Computing macroscopic stress-strain relations for
  one-dimensional chains of atoms}
\label{sec:static}

In this section, we wish to compute the stress-strain relation of a
one-dimensional chain of atoms, in the thermodynamic limit. More
precisely, we consider a chain of $1+N$ atoms, with its left-end atom fixed,
and either submit the right-end atom to a force, and compute the average
elongation, or prescribe the elongation, and compute the force. We will
show that, in the limit $N \to \infty$, these two relations are
identical, and that they can be computed in an extremely efficient
manner. In short, passing to the limit $N \to \infty$ makes tractable a
computation that is, for finite and large $N$, very expensive. 

The relation between that question and the question of determining the
free energy of the system, when the reaction coordinate is the length of
the system, will also be discussed. 

In the sequel, we first proceed with the nearest neighbour case (see
Section~\ref{sec:NN}). We next address the next-to-nearest neighbour
case in Section~\ref{sec:NNN}, which is technically more involved.

\subsection{The nearest neighbour (NN) case}
\label{sec:NN}

We consider a one-dimensional chain of atoms, with positions $q^0$,
$q^1$, \ldots, $q^N$. In this section, we only consider nearest neighbour
interaction. In addition to this internal interaction, we assume that
the atom at the right boundary of the chain is submitted to an external
force $f$, and that the atom at the left boundary is fixed: $q^0 = 0$. The
energy of the chain thus reads
$$
\widetilde{E}_f \left( q^1,\ldots,q^N \right) =
\sum_{i=1}^{N} W \left( q^{i} - q^{i-1} \right)
- f q^N.
$$
In the sequel, we will consider the limit when the number $N$ of atoms
goes to $\infty$. We wish to make sure that, even when $N \to \infty$,
the system occupies, on average, a finite length. To this aim, we
introduce the rescaled positions $u^i = h q^i$, with $h = 1/N$. The
energy now reads 
\begin{equation}
\label{eq:energie-NN-stress}
E_f \left( u^1,\ldots,u^N \right) =
\sum_{i=1}^{N} W \left( \frac{u^{i} - u^{i-1}}{h} \right)
- f \frac{u^N}{h}
\end{equation}
where again $u^0 = 0$. 

For any observable $\Phi$, depending on the variables $u^1,\ldots,u^N$, we
define the canonical average of $\Phi$ by
\begin{equation}
\label{eq:can_aver}
\langle \Phi \rangle_N^f
=
Z^{-1} \int_{\rens^N} \Phi \left( u^1,\ldots,u^N \right)
\exp \left( -\beta E_f \left( u^1,\ldots,u^N \right) \right)
du^1 \ldots \, du^N,
\end{equation}
where the partition function $Z$ reads
$$
Z = \int_{\rens^N} 
\exp \left( -\beta E_f \left( u^1,\ldots,u^N \right) \right)
du^1 \ldots \, du^N.
$$
We assume in the sequel that $W(r)$ grows fast enough to $\infty$ when
$|r| \to \infty$, so that $Z$ is well defined (it is for instance enough
that $W(r) \sim_{|r| \to \infty} |r|^\alpha$ with $\alpha > 1$). 

We will be interested in the limit of $\langle \Phi \rangle_N^f$, when
$N \to \infty$, and when $\Phi$ only depends on $u^N$:
$\Phi(u^1,\ldots,u^N) = A(u^N)$ for a given function $A$.  

\begin{remark}
In (\ref{eq:can_aver}), we let the variables $u^i$ vary on the whole
real line. We do not constrain them to obey $u^{i-1}\leq
u^{i}$, which would encode the fact that nearest neighbours remain
nearest neighbours. The argument provided here carries
through when this constraint is accounted for: we just need to replace 
the interaction potential $W$ by
$$
W_c(y)=\left\{
\begin{array}{ll}
W(y)&\hbox{\rm when}\quad y\geq 0\\
+\infty&\hbox{\rm otherwise.}
\end{array}\right.
$$
\end{remark}

\subsubsection{Computing the strain for a given stress}
\label{sec:NN_y}

We first show a simple adaptation of~\cite[Theorem~1]{bllp}, which is
useful to compute averages of general observables, in the
thermodynamic limit, for the canonical ensemble at a fixed stress:

\begin{lemma}
\label{lem:NN_y}
Assume that $A: \rens \to \rens$ is continuous, that for some $p \geq
1$, there exists a constant $C$ such that
$$
\forall y \in \rens, \quad
\left| A(y) \right| \leq C \left(1+|y|^p\right),
$$ 
and that
$$
\int_\rens \left(1+|y|^p \right)
\exp\left( -\beta \left[ W(y) - fy \right] \right) dy <
+\infty.
$$
Then
$$
\lim_{N \to \infty} \langle A (u^N) \rangle_N^f = A\left(y^\star(f)\right),
$$
with
\begin{equation}
\label{eq:def_y_star}
y^\star(f) = 
\frac{
\int_\rens y \ \exp(-\beta \left[ W(y) - fy \right] ) \, dy
}{
\int_\rens \exp(-\beta \left[ W(y) - fy \right]) \, dy
}.
\end{equation}
\end{lemma}

\begin{proof}
\smartqed
We observe that 
\begin{eqnarray*}
\langle A \rangle_N^f &=& Z^{-1} \int_{\rens^N} 
A\left(u^N\right)
\exp \left( -\beta E_f \left( u^1,\ldots,u^N \right) \right)
du^1 \ldots \, du^N
\\
&=&
Z^{-1} \int_{\rens^N} 
A\left( u^N \right)
\exp \left( -\beta \sum_{i=1}^{N} W_f \left( \frac{u^i - u^{i-1}}{h} \right) \right)
du^1 \ldots \, du^N
\end{eqnarray*}
where $W_f(x) = W(x) - fx$. Introducing $\dis y^i = \frac{u^i -
  u^{i-1}}h$, a change of variables in the above integral yields
$$
\langle A \rangle_N^f
=
Z^{-1} \int_{\rens^N} 
A \left( \frac1N \sum_{i=1}^N y^i \right)
\exp \left( -\beta \sum_{i=1}^{N} W_f \left(y^i \right) \right)
dy^1 \ldots \, dy^N
$$
where, with a slight abuse of notation, $\dis 
Z = \int_{\rens^N} \exp \left( -\beta \sum_{i=1}^{N} W_f \left(y^i \right) \right)
dy^1 \ldots \, dy^N$. Consider now a sequence $\left\{ Y^i
\right\}_{i=1}^N$ of independent random variables, sharing the same law 
$z^{-1} \exp \left( -\beta W_f(y) \right) dy$ with
$\dis z = \int_\rens \exp \left( -\beta W_f(y) \right) dy$. It is clear that
$$
\langle A \rangle_N^f
=
\EE \left[ A \left( \frac1N \sum_{i=1}^N Y^i \right) \right].
$$
The law of large numbers readily yields that $\dis \frac1N \sum_{i=1}^N Y^i$
converges almost surely to $y^\star(f)$ defined by~\eqref{eq:def_y_star}. 

We infer from~\cite[Theorem 1]{bllp} that,
for any force $f$, and for any observable $A$ sufficiently smooth,
the limit when $N \to \infty$ of $\langle A \rangle_N^f$ is
$$
\lim_{N \to \infty} \langle A \rangle_N^f = A(y^\star(f)).
$$ 
Rates of convergence are also provided in the
same theorem. 
\qed
\end{proof}

Numerical simulations
illustrating this result are reported in~\cite[Section 2.3]{bllp}.

\medskip

In the specific case of interest here, namely computing the
stress-strain relation, we take $A(u^N) = u^N$, thus $\eps_N(f) :=
\langle A \rangle_N^f$ represents the average length of the chain, for a
prescribed force $f$. We infer from the previous result that 
$$
\lim_{N \to \infty} \eps_N(f) = y^\star(f).
$$ 
We hence have determined the macroscopic elongation, namely $y^\star(f)$, for a
prescribed microscopic force $f$ in the chain. 

Notice that, in this specific case, $A$ is a linear function, so we
actually have
$\eps_N(f) = y^\star(f)$ for any $N$. The result of Lemma~\ref{lem:NN_y} remains
interesting for computing standard deviation of the average length, for
example. 

\begin{remark}
The force between atoms $j$ and $j-1$ is 
$\dis W'\left( \frac{u^j - u^{j-1}}h \right)$. Its canonical
average, defined by~\eqref{eq:can_aver}, is 
\begin{eqnarray*}
\sigma^j_N &=& Z^{-1} \int_{\rens^N} 
W'\left( \frac{u^j - u^{j-1}}h \right)
\exp \left( -\beta E_f \left( u^1,\ldots,u^N \right) \right)
du^1 \ldots \, du^N
\\
&=&
Z^{-1} \int_{\rens^N} W'\left(y^j\right)
\exp \left( -\beta \sum_{i=1}^{N} \left[ W\left(y^i\right) - f y^i \right] \right)
dy^1 \ldots \, dy^N
\\
&=&
\frac{
\int_\rens W'\left(y^j\right)
\exp \left( -\beta \left[ W\left(y^j\right) - f y^j \right] \right)
dy^j
}{
\int_\rens \exp \left( -\beta \left[ W\left(y^j\right) - f y^j \right] \right)
dy^j
}
\\
&=&
f + \frac{
\int_\rens \left[ W'\left(y^j\right) - f \right]
\exp \left( -\beta \left[ W\left(y^j\right) - f y^j \right] \right)
dy^j
}{
\int_\rens \exp \left( -\beta \left[ W\left(y^j\right) - f y^j \right] \right)
dy^j
}
\end{eqnarray*}
where $\dis y^j = \frac{u^j - u^{j-1}}h$. Integrating by parts, we see
that the second term of the last line vanishes. We hence obtain that the
average force 
between two consecutive atoms is independent of $j$ (the stress is
homogeneous in the material), and is equal to its prescribed microscopic
value $f$:
$$
\forall j, \ \ \forall N, \quad \sigma_N^j = f.
$$
Imposing a force $f$ on the right boundary atom hence implies that the
average force between any two consecutive atoms is equal to $f$. 
$\diamond$
\end{remark}

\subsubsection{Computing the stress for a given strain}
\label{sec:NN_f}

In the previous section, we have prescribed a force, and computed an
average elongation. 
We now prescribe the length of the material, by imposing $u^0=0$ and
$u^N = x$ (see Fig.~\ref{fig:chain}). 

\begin{figure}[htbp]
\psfrag{0}{0}
\psfrag{x}{$x$}
\psfrag{N}{$N$}
\centerline{
\includegraphics[width=6cm]{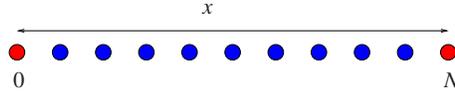}
}
\caption{One-dimensional chain of $1+N$ atoms, where the total length of
  the system is prescribed at the value $x$. 
}
\label{fig:chain}
\end{figure}

As we fix the position of atom $N$, the system is insensitive
to any force $f$ imposed on that atom. We hence set $f=0$. Our aim
is to compute the force in the chain, 
\begin{equation}
\label{eq:def_TN_init}
{\cal T}_N(x) = \frac{\dis{
\int_{\rens^{N-1}} 
W'\left( \frac{x - u^{N-1}}h \right) 
\exp \left(-\beta
E_0(u^1,\ldots,u^{N-1},x) \right) \ du^1 \dots du^{N-1}
}}{\dis{
\int_{\rens^{N-1}} \exp \left(-\beta
E_0(u^1,\ldots,u^{N-1},x) \right) \ du^1 \dots du^{N-1}
}},
\end{equation}
or, more precisely, its limit when $N \to \infty$. Note that, as all the 
$\dis \frac{u^i - u^{i-1}}h$ play the same role in the above expression, we
also have, for any $1 \leq i \leq N-1$,
$$
{\cal T}_N(x) = \frac{\dis{
\int_{\rens^{N-1}} 
W'\left( \frac{u^i - u^{i-1}}h \right) 
\exp \left(-\beta
E_0(u^1,\ldots,u^{N-1},x) \right) \ du^1 \dots du^{N-1}
}}{\dis{
\int_{\rens^{N-1}} \exp \left(-\beta
E_0(u^1,\ldots,u^{N-1},x) \right) \ du^1 \dots du^{N-1}
}}.
$$
The force between atom $N$ and $N-1$ is thus equal to the force between
any two consecutive atoms. 

\medskip

We infer from~\eqref{eq:def_TN_init} that ${\cal T}_N(x) = F'_N(x)$, where
$$
F_N(x) = -\frac 1 {\beta N} \ln \left[ \int_{\rens^{N-1}}\exp\left(-\beta
E_0(u^1,\ldots,u^{N-1},x) \right) \ du^1 \dots du^{N-1} \right].
$$
Hence $N F_N$ is the free energy of the material associated to the reaction coordinate $\xi(u^1,\ldots,u^N)=u^N$, and $F_N$ is a rescaled free energy
(free energy per integrated out particle). 
Using the variables $\dis y^i
= \frac{u^i - u^{i-1}}h$, we also see 
that $\exp(-\beta N F_N(x)) \, dx $ is (up to a normalizing multiplicative constant) the probability distribution of the random variable 
$\dis \frac1N \sum_{i=1}^N Y^i$, when $\left\{ Y^i
\right\}_{i=1}^N$ is a sequence of independent random variables, sharing
the same law $z^{-1} \exp \left( -\beta W(y) \right) dy$, with 
$\dis z = \int_\rens \exp \left( -\beta W(y) \right) dy$.

In the case 
$\dis W(y) = \frac12 (y-a)^2$, it is possible to analytically compute
$F_N(x)$, and to observe that there exists a
constant $C_N$, independent of $x$, such that $F_N(x)+C_N$ has a finite
limit when $N \to \infty$.
In the general case, the limit of $F_N$ is given by the following
result, which relies on a 
large deviations result for i.i.d. random variables:

\begin{lemma}[\cite{bllp}, Theorem~2]
\label{lem:LDP_NN}
Assume that the potential $W$ satisfies
$$
\forall \xi \in \rens, \quad \int_\rens \exp \left(\xi y -
\beta W(y)\right) dy < +\infty,
$$
and $\exp(-\beta W) \in H^1(\rens)$. Then 
\begin{equation}\label{eq:cv_Finfty}
\lim_{N \to +\infty}
\left( F_N(x) + \frac{1}{\beta} \ln \frac{z}{N} \right) 
= F_\infty(x)
\end{equation}
with
\begin{equation}
\label{eq:def_Finfty}
F_\infty(x) := 
\frac{1}{\beta} \sup_{\xi \in \rens} \left( 
\xi x - \ln \left[ z^{-1} \int_{\rens} \exp( \xi y - \beta W(y) ) \, dy 
\right] \right)
\end{equation}
and $\displaystyle{ z = \int_\rens \exp(-\beta W(y)) \, dy }$. This
convergence holds pointwise in $x$, and also in $L^p_{\rm loc}$, for any
$1 \leq p < \infty$. As a consequence, $F'_N$ converges to $F'_\infty$
in $W^{-1,p}_{\rm loc}$. 
\end{lemma}

We hence obtain the macroscopic force $F_\infty'(x)$ for a prescribed
elongation $x$. Numerical simulations that illustrate
this result are reported in~\cite[Section 2.3]{bllp}.

\begin{remark}
The additive term $\beta^{-1} \ln(z/N)$ in~\eqref{eq:cv_Finfty} can be
seen as a normalizing constant. Indeed, as mentioned above, $N F_N$ is a
free energy, and the correct normalization for $\exp(-\beta N F_N)$ to
be a probability density function is: 
$$
\int_{\rens} \exp\left[-\beta N \left( F_N(x) + \frac{1}{\beta} \ln
    \frac{z}{N} \right) \right] \, dx =1.
$$ 
$\diamond$
\end{remark}

\begin{remark}
$F_N$ is a challenging quantity to compute. One possible method is to
compute, for each $x$, its derivative $F_N'(x)$, and deduce
$F_N$ (this is the so-called thermodynamic integration method). Note
that $F_N'(x) = {\cal T}_N(x)$ is given by~\eqref{eq:def_TN_init}: it is
a canonical average of some observable, in a space of dimension $N-1 \gg
1$. In contrast, $F_\infty$ is easier to compute, since it only involves
one-dimensional integrals or optimization problems. 
$\diamond$
\end{remark}

\subsubsection{Equivalence of stress-strain relations in the
  thermodynamic limit}
\label{sec:equiv_NN}

The function we maximize in~\eqref{eq:def_Finfty} is concave, so there
exists a unique maximizer $\xi(x)$ in~\eqref{eq:def_Finfty}, that 
satisfies the Euler-Lagrange equation
\begin{equation}
\label{eq:toto1}
x = \frac{
\int_{\rens} y \ \exp( \xi(x) y - \beta W(y) ) \, dy
}{
\int_{\rens} \exp( \xi(x) y - \beta W(y) ) \, dy 
}.
\end{equation}
We observe that 
$$
F_\infty'(x) = \frac{\xi(x)}{\beta}.
$$
On the other hand, recall the definition~\eqref{eq:def_y_star} of $y^\star(f)$:
$$
y^\star(f) = 
\frac{
\int_\rens y \ \exp(-\beta \left[ W(y) - fy \right] ) \, dy
}{
\int_\rens \exp(-\beta \left[ W(y) - fy \right]) \, dy
}.
$$
Comparing~\eqref{eq:toto1} and~\eqref{eq:def_y_star}, we see that
$y^\star(\beta^{-1} \xi(x)) = y^\star(F_\infty'(x)) = x$. The function $f \mapsto
y^\star(f)$ is increasing (because its derivative is positive), thus it is
injective, and we also get the 
converse relation: $F_\infty'(y^\star(f)) = f$. 

Otherwise stated, the relation $f \mapsto y^\star(f)$ and 
$x \mapsto F'_\infty(x)$ are inverse one to each other. So,
prescribing a microscopic force $f$ and computing the macroscopic
elongation is equivalent to prescribing an elongation and computing the
macroscopic force, {\em in the thermodynamic limit} (namely in the limit
$N \to \infty$). 

\subsection{The next-to-nearest neighbour (NNN) case}
\label{sec:NNN}

We now consider next-to-nearest neighbour interactions in the
chain. Again, the first atom is fixed: $u^0 = 0$, whereas the last one is
submitted to an external force $f$. The (rescaled) energy reads
\begin{equation}
\label{eq:energie-NNN-stress}
E_f \left( u^1,\ldots,u^N \right) =
\sum_{i=1}^{N} W_1\left( \frac{u^{i} - u^{i-1}}{h} \right)
+
\sum_{i=1}^{N-1} W_2 \left( \frac{u^{i+1} - u^{i-1}}{h} \right)
- f \frac{u^N}{h}.
\end{equation}
If $W_2 \equiv 0$, this energy reduces
to~\eqref{eq:energie-NN-stress}. Averages of observables are again defined
by~\eqref{eq:can_aver}. 

\subsubsection{Computing the strain for a given stress}
\label{sec:NNN_y}

Our aim, as in Section~\ref{sec:NN_y}, is to compute the macroscopic
strain, which is the average length of the material, that is
$$
\eps_N(f) = \langle u^N \rangle_N^f,
$$
where $\langle \cdot \rangle_N^f$ is the average with respect
to the canonical measure associated to $E_f$. We introduce the notation
$$
W_{1f}(x) = W_1(x) - fx,
$$
which will be useful in the sequel. A simple adaptation
of~\cite[Theorem 3]{bllp} yields the following general result: 
\begin{lemma}
\label{lem:NNN_y}
Assume that $A : \rens \mapsto \rens$ is continuous, and that there
exists $p \geq 1$ and $C>0$ such that
$$
|A(x)| \leq C(1+|x|^p).
$$
Assume also that $W_{1f}$ and $W_2$ both belong to $L^1_{\rm
  loc}(\rens)$, that they are bounded from below, and that, for any $x
\in \rens$, we have $| W_{1f}(x) | < \infty$ and $| W_2(x) | < \infty$.
In addition, we assume that
$e^{-\beta W_{1f}}$ and $e^{-\beta W_2}$ both belong to $W^{1,1}_{\rm
  loc}(\rens)$, with
$$
\int_\rens (1+|x|^p) \ e^{-\beta W_{1f}(x)} dx < +\infty
\quad \text{and} \quad 
\int_\rens (1+|x|^p) \ e^{-\beta W_2(x)}dx <+\infty.
$$
Then 
\begin{equation}
\label{eq:lim}
\lim_{N \to \infty} \langle A(u^N) \rangle_N^f = A(y^\star(f))
\end{equation}
with
\begin{equation}
\label{eq:def_ystar_NNN}
y^\star(f) = \int_\rens y \ \psi_f^2(y) \, dy,
\end{equation}
where $\psi_f$ solves the variational problem
\begin{equation}
\label{eq:psi1}
\lambda_f = 
\max_{\psi \in L^2(\rens)}
\left\{ \int_{\rens^2} \psi(y) \ \psi(z) \ K_f(y,z) \ dy \, dz; \
\int_\rens \psi^2(y) \, dy = 1 \right\},
\end{equation}
with
\begin{equation}
\label{eq:K_f}
K_f(x,y) := \exp\left[-\beta W_2(x+y) -\frac \beta 2 W_{1f}(x) -
\frac \beta 2 W_{1f}(y) \right].
\end{equation}
\end{lemma}

We only provide here the main arguments to prove this result
(see~\cite[Sec. 3.1.1 and Theorem 3]{bllp} for details). They will be
useful in the sequel. The observable $A(u^N)$ only depends on $u^N$,
thus 
\begin{eqnarray*}
\langle A(u^N) \rangle_N^f &=& Z^{-1} \int_{\rens^N} 
A\left(u^N\right)
\exp \left( -\beta E_f \left( u^1,\ldots,u^N \right) \right)
du^1 \ldots \, du^N
\\
&=&
Z^{-1} \int_{\rens^N} 
A\left(u^N\right)
\exp \left( -\beta \sum_{i=1}^{N} W_{1f} \left( \frac{u^i - u^{i-1}}{h}
  \right) 
\right. \\ && \hspace{2cm} \left.
- \beta \sum_{i=1}^{N-1} W_2 \left( \frac{u^{i+1} - u^{i-1}}{h} \right)
\right)
du^1 \ldots \, du^N.
\end{eqnarray*}
Introducing again the variables
$\dis y^i = \frac{u^i - u^{i-1}}h$, we see that 
\begin{equation}
\label{eq:imp}
\langle A(u^N) \rangle_N^f =
Z^{-1} \int_{\rens^N} 
A\left( \frac1N \sum_{i=1}^N y^i \right) 
\exp \left( -\beta W_{1f} \left( y^1 \right) \right)
\prod_{i=2}^{N} k_f \left(y^{i-1},y^i \right) \
dy^1 \dots dy^N
\end{equation}
with
$$ 
k_f \left(y^{i-1},y^i \right) 
=
\exp \left( -\beta W_{1f} \left( y^i \right)
- \beta W_2 \left( y^{i-1} + y^i \right)
\right)
$$
Assume for a moment that $\dis \int_\rens k_f(a,b) \, db =1$. Then we see
that  
$$
\langle A(u^N) \rangle_N^f =
\EE \left[ A \left( \frac1N \sum_{i=1}^N Y^i \right) \right],
$$
where $\left\{ Y^i \right\}_{i=1}^N$ is a realization of a Markov chain
of transition kernel $k_f$, and where $Y^1$ has the initial law (up to a
normalization constant)
$\exp \left( -\beta W_{1f} \left( y^1 \right) \right) \, dy^1$. A law of
large numbers argument, now for Markov chains, yields the large $N$ limit
of $\langle A(u^N) \rangle_N^f$ (recall that, in the case of the NN model considered in
Section~\ref{sec:NN_y}, this limit is given
by a law of large numbers argument for i.i.d. sequences).

In general, of course, $\dis \int_\rens k_f(a,b) \, db \neq 1$. There is
thus a slight technical difficulty in identifying a Markov chain
structure in~\eqref{eq:imp}. It yet
turns out that the above argument can be made rigorous as follows. 
Consider the variational problem~\eqref{eq:psi1}, with $K_f$ defined
by~\eqref{eq:K_f}. Under our assumptions, 
$K_f \in L^2(\rens \times \rens)$. Using standard tools of spectral
theory of self-adjoint operators (see {\em e.g.}~\cite{dunford}), one
can prove that this problem has a maximizer (denoted $\psi_f$), and
that, up to changing $\psi_f$ in $-\psi_f$, the maximizer is unique. In
addition, one can choose it such that  $\psi_f>0$. We can next define
\begin{equation}
\label{eq:def_g}
g_f(x,y) := \frac{\psi_f(y)}{\lambda_f \psi_f(x)} \ K_f(x,y),
\end{equation}
which satisfies 
$$
\int_\rens g_f(y,z) \, dz = 1, \quad
\int_\rens \psi_f^2(y) \ g_f(y,z) \, dy = \psi_f^2(z).
$$
The average~\eqref{eq:imp} now reads
\begin{equation}
\label{eq:using_g}
\begin{array}{c}
\dis \hspace{-2cm}
\langle A(u^N) \rangle_N^f = 
Z_g^{-1} \int_{\rens^N}
A \left( \frac{1}{N} \sum_{i=1}^N y^i \right) \,
\psi_f(y^1) \ e^{-\frac \beta 2 W_{1f}(y^1)} 
\\ 
\dis \hspace{2.5cm}
\times g_f(y^1,y^2) \ldots g_f(y^{N-1},y^N) \
\frac{e^{-\frac \beta 2 W_{1f}(y^N)}}{\psi_f(y^N)} \ 
dy^1 \ldots dy^N,
\end{array}
\end{equation}
with $\displaystyle{
Z_g = \int_{\rens^N} 
\psi_f(y^1) \ e^{-\frac \beta 2 W_{1f}(y^1)} \
g_f(y^1,y^2) \ldots g_f(y^{N-1},y^N)  \
\frac{e^{-\frac \beta 2 W_{1f}(y^N)}}{\psi_f(y^N)} \ dy^1 \ldots dy^N
}$. Thus 
$$
\langle A(u^N) \rangle_N^f =
\EE \left[ A \left( \frac1N \sum_{i=1}^N Y^i \right) \right],
$$
where
$(Y^1,\ldots,Y^N)$ may now be seen as a realization of a \emph{normalized}
Markov chain of kernel $g_f$, with invariant probability measure
$\psi_f^2$. 

Under our assumptions, the Markov chain has a unique invariant measure,
and satisfies a law of large numbers with respect to it. This yields the
convergence~\eqref{eq:lim}. 
Numerical simulations
illustrating this result are reported in~\cite[Section 3.1.3]{bllp}.

\bigskip

In the specific case of interest here, namely computing the
stress-strain relation, we take $A(u^N) = u^N$, thus $\eps_N(f) :=
\langle A \rangle_N^f$ represents the average length of the chain, for a
prescribed force $f$. We infer from the previous result that 
$$
\lim_{N \to \infty} \eps_N(f) = y^\star(f).
$$ 
We hence have determined the macroscopic elongation, namely $y^\star(f)$, for a
prescribed microscopic force $f$ in the chain. 

\bigskip

We conclude this section by showing the following result, which will be
useful in the sequel. 

\begin{lemma}
\label{lem:croissance}
Under the assumptions of Lemma~\ref{lem:NNN_y}, introduce the asymptotic
variance~$\sigma^2(f)$ defined by
\begin{equation}
\label{eq:def_sigma_f}
\sigma^2(f) = \int_\rens (x-y^\star(f))^2 \ \psi_f^2(x) \, dx 
+ 2 \sum_{i\geq 2} \EE\left( (\widetilde Y_i - y^\star(f))(\widetilde Y_1 -
y^\star(f))\right)
\end{equation}
where $\left( \widetilde Y_i\right)_{i \geq 1}$ is a Markov chain of
transition kernel $g_f$, and of initial law $\psi_f^2$, the invariant
measure. 

Assume that $\sigma^2(f) \neq 0$ almost everywhere. Then the function $f
\mapsto y^\star(f)$ is increasing. 
\end{lemma}
Note that the right-hand side of~\eqref{eq:def_sigma_f}
is exactly the variance appearing in the Central Limit
Theorem for Markov chains~\cite[Theorem~17.0.1]{meyn}. It is thus
non-negative. More precisely, we have that
$\dis \lim_{N \to \infty}
N \, \Var \left( \frac1N \sum_{i=1}^N \widetilde Y_i \right) 
=
\sigma^2(f)
$
where $\left( \widetilde Y_i\right)_{i \geq 1}$ is the Markov chain
defined in the above lemma.

\begin{proof}
\smartqed
Let $\eps_N(f) := \langle u^N \rangle_N^f$. An analytical computation
shows that
$$
D_N(f) := 
\frac{d\eps_N}{df}(f) 
= N \beta \left[
\langle (u^N)^2 \rangle_N^f - \left( \langle u^N \rangle_N^f \right)^2
\right].
$$ 
Thus the function $f \mapsto \eps_N(f)$ is non-decreasing. By
Lemma~\ref{lem:NNN_y}, 
$y^\star(f)$ is the pointwise limit of $\eps_N(f)$: it is thus
non-decreasing. It remains to prove that it is increasing.

Let us now compute the limit when $N \to \infty$ of
$D_N(f)$. Using~\cite[Theorem~4]{bllp}, we see that
$$
\lim_{N \to \infty} D_N(f) = \beta \sigma^2(f),
$$
where $\sigma^2(f)$ is defined by~\eqref{eq:def_sigma_f}.

Let us now fix $\tau$ and $\overline{\tau} \geq \tau$. Since
$D_N(f) \geq 0$, we can use Fatou lemma, which yields that
$$
\beta \int_\tau^{\overline{\tau}} \sigma^2(f) df 
=
\int_\tau^{\overline{\tau}} \liminf D_N(f) \, df
\leq
\liminf \int_\tau^{\overline{\tau}} D_N(f) \, df
=
y^\star(\overline{\tau}) - y^\star(\tau).
$$
As $\sigma^2(f) > 0$ almost everywhere, we thus obtain that $\tau \mapsto
y^\star(\tau)$ is an increasing function.
\qed
\end{proof}

\subsubsection{Computing the stress for a given strain}
\label{sec:NNN_f}

We now prescribe the length of the material, by imposing $u^0=0$ and
$u^N = x$. Our aim is to compute the average force in the chain,
\begin{equation}
\label{eq:def_TN}
{\cal T}_N(x) = \frac{\dis{
\int_{\rens^{N-1}} 
A_h\left(u^{N-1},u^{N-2};x\right) 
\exp \left(-\beta
E_0\left(u^1,\ldots,u^{N-1},x\right) \right) \ du^1 \dots du^{N-1}
}}{\dis{
\int_{\rens^{N-1}} \exp \left(-\beta
E_0\left(u^1,\ldots,u^{N-1},x\right) \right) \ du^1 \dots du^{N-1}
}},
\end{equation}
where $E_0$ is the energy~\eqref{eq:energie-NNN-stress} with $f=0$, and 
where the observable $A_h$ is the force acting at the end
of the chain, which reads
$$
A_h(u^{N-1},u^{N-2};x) = W_1' \left( \frac{x - u^{N-1}}{h} \right)
+
W_2' \left( \frac{x - u^{N-2}}{h} \right).
$$
More precisely, we are interested in 
$\dis \lim_{N \to \infty} {\cal T}_N(x)$. 

As in Section~\ref{sec:NN_f}, we see that ${\cal T}_N(x) = F_N'(x)$,
with 
\begin{equation}
\label{eq:FN_NNN}
F_N(x) = -\frac 1 {\beta N} \ln \left[ \int_{\rens^{N-1}}\exp\left(-\beta
E_0(u^1,\ldots,u^{N-1},x) \right) \ du^1 \dots du^{N-1} \right].
\end{equation}
Again, $N F_N$ is the free energy associated to the reaction coordinate $\xi(u^1,\ldots,u^N)=u^N$, and $F_N$ is a rescaled free energy
(free energy per integrated out particle). In the NN case, we have
computed the large $N$ limit of $F_N(x)$ using a large deviations result
for i.i.d. random variables. Comparing Sections~\ref{sec:NN_y}
and~\ref{sec:NNN_y}, we also see that moving from a NN setting to a NNN
setting implies moving from a framework where random variables are 
i.i.d. to a framework where they are a
realization of a Markov chain. It is hence natural to try and use a
large deviations result for Markov chains to compute the large $N$ limit
of~\eqref{eq:FN_NNN}.

We now assume that the underlying Markov chain satisfies the following
{\em pointwise} large deviations result:
\begin{assum}
\label{assum:LDP_MC}
Consider the Markov chain $\left\{ Y^i \right\}_{i \geq 1}$ of kernel
$k \in L^2(\rens \times \rens)$. Assume that, for any $\xi \in \rens$, the
function $\exp(\xi y) \, k(x,y) \in L^2(\rens \times \rens)$.

Introduce the operator (on $L^2(\rens)$)
$$
(Q_\xi \varphi)(y) = \int_\rens \varphi(x) \, \exp(\xi y) \, k(x,y) \, dx
$$
and assume that it has a simple and isolated largest eigenvalue
$\Lambda(\xi)$, and that $\xi \mapsto \ln \Lambda(\xi)$ is convex.

Let $\exp(-N \overline{F}_N(x)) \, dx $ be the law of the random
variable $\dis \frac1N \sum_{i=1}^N Y^i$. We assume the large deviations
principle
\begin{equation}
\label{eq:LDP_NNN}
\lim_{N \to +\infty} \overline{F}_N(x) 
= \overline{F}_\infty(x)
\end{equation}
where 
\begin{equation}
\label{eq:def_Finfty_NNN}
\overline{F}_\infty(x) := 
\sup_{\xi \in \rens} \left( 
\xi x - \ln \Lambda(\xi) \right).
\end{equation}
We moreover assume that 
the convergence~\eqref{eq:LDP_NNN} holds pointwise in $x$, and also in
$L^p_{\rm loc}$, for any $1 \leq p < \infty$.
As a consequence, $\overline{F}'_N$ converges to $\overline{F}'_\infty$ in 
$W^{-1,p}_{\rm loc}$. 
\end{assum}

Note that similar results in a finite state Markov chain setting are reviewed
in~\cite[pages 60--61]{hollander} or~\cite[Sec. 3.1.1]{dembo}
(the continuous state case is addressed in {\em e.g.}~\cite[Secs. 6.3
and 6.5]{dembo}). In the discrete state case, one can prove that $\xi
\mapsto \ln \Lambda(\xi)$ is convex (see~\cite[Exercise
V.14]{hollander}). We will numerically check in the sequel that this
assumption is indeed satisfied in the example we consider (see
Fig.~\ref{fig:ln_lambda_xi}).  

\begin{remark}
We have assumed that the operator $Q_\xi$ has a simple and isolated
largest eigenvalue. This can be proved for many kernels $k$, using for
instance Krein-Rutman theorem~\cite{Schaefer99}. In the case of interest
in this contribution, we will use the specific expression of the kernel
to transform the operator $Q_\xi$ into a self-adjoint Hilbert-Schmidt
operator on $L^2(\rens)$ (see Remark~\ref{rem:aussi} below). We will
thus be in position to work with self-adjoint compact operators.
$\diamond$
\end{remark}

\begin{remark}
In the NN case, when $k(x,y) = \theta(y) = z^{-1} \exp (-\beta W(y))$,
the sequence $\left\{ Y^i 
\right\}_{i \geq 1}$ is a sequence of i.i.d. variables sharing the same
law $\theta(y) \, dy$. The operator $Q_\xi$ has a unique eigenvalue
$\dis \Lambda(\xi) = \int_\rens \exp(\xi y) \, \theta(y) \, dy$.
We then recover
the large deviations result of i.i.d. sequence given in
Lemma~\ref{lem:LDP_NN} 
(see {\em e.g.}~\cite{Ell85a,Ell85b,Ell95,varadhan}). 
$\diamond$
\end{remark}

We now wish to use Assumption~\ref{assum:LDP_MC} to compute the large $N$ limit
of~\eqref{eq:FN_NNN}. As pointed out in Section~\ref{sec:NNN_y}, there
is a slight technical difficulty in identifying a Markov chain structure
in the
NNN setting, related to the normalization of the Markov chain kernel. We
thus cannot readily use Assumption~\ref{assum:LDP_MC}. We now detail how to
overcome this difficulty.  

Consider an observable $A$ that depends only on
$u_N$. In view of~\eqref{eq:FN_NNN} and~\eqref{eq:using_g}, its
canonical average reads 
\begin{eqnarray*}
\langle A \rangle_N &=& Z^{-1} \int_{\rens^N} A\left(u^N\right) 
\exp \left(-\beta
E_0\left(u^1,\ldots,u^{N-1},u^N\right) \right) \ du^1 \dots du^{N-1} \, du^N
\\
&=&
Z^{-1} \int_{\rens} A(x) \exp \left(-\beta N F_N(x) \right) \, dx
\\
&=&
Z_g^{-1} \int_{\rens^N}
A \left( \frac{1}{N} \sum_{i=1}^N y^i \right) \,
\psi_0(y^1) \ e^{-\frac \beta 2 W_1(y^1)} 
\\
&& \hspace{2.5cm}
\times g_0(y^1,y^2) \ldots g_0(y^{N-1},y^N) \
\frac{e^{-\frac \beta 2 W_1(y^N)}}{\psi_0(y^N)} \ 
dy^1 \ldots dy^N,
\end{eqnarray*}
where $g_0$ is defined by~\eqref{eq:def_g} and $\psi_0$ is the maximizer
in~\eqref{eq:psi1}, when the body force~$f = 0$.
Let ${\cal P}(y^1,\ldots,y^N)$ be the probability density of a Markov
chain $\left\{ Y^i \right\}_{i=1}^N$ of kernel~$g_0$, where the law of
$Y^1$ is (up to a normalization constant) $\psi_0(y^1) \ e^{-\frac \beta
  2 W_1(y^1)} \, dy^1$. Then
\begin{equation}
\label{eq:tutu}
\int_{\rens} A(x) \exp \left(-\beta N F_N(x) \right) \ dx
=
C_N \int_{\rens^N}
A \left( \frac{1}{N} \sum_{i=1}^N y^i \right) \,
{\cal P}(y^1,\ldots,y^N) \,
r(y^N) \,
dy^1 \ldots dy^N
\end{equation}
where $C_N$ is a constant that does not depend on the observable $A$, and 
$$
r(y^N) = \frac{e^{-\frac \beta 2 W_1(y^N)}}{\psi_0(y^N)} .
$$
Let now $\alpha_N(x,y^N) \, dx d y^N$ be the law of the couple 
$\dis \left( \frac{1}{N} \sum_{i=1}^N Y^i, Y^N \right)$. We
recast~\eqref{eq:tutu} as
$$
\int_{\rens} A(x) \exp \left(-\beta N F_N(x) \right) \ dx
=
C_N \int_{\rens^2}
A \left( x \right) \, \alpha_N \left(x,y^N\right) \, 
r\left(y^N\right) \, dx \, dy^N.
$$
As this relation holds for any observable $A$, with a constant $C_N$
independent of $A$, we obtain
$$
\exp \left(-\beta N F_N(x) \right)
=
C_N \int_{\rens} \alpha_N\left(x,y^N\right) \, r\left(y^N\right) \, dy^N.
$$
Assuming that $r$ and $1/r$ are in $L^\infty(\rens)$, we have
$$
C_N \| 1/r \|_{L^\infty}^{-1} \int_{\rens} \alpha_N \left(x,y^N \right) \, dy^N
\leq
\exp \left(-\beta N F_N(x) \right)
\leq
C_N \| r \|_{L^\infty} \int_{\rens} \alpha_N\left(x,y^N\right) \, dy^N.
$$
As a consequence, since the function $r$ is independent of $N$,
\begin{equation}
\label{eq:lim1}
\lim_{N \to \infty} \left( F_N(x) + D_N \right) 
=
\lim_{N \to \infty} \left[ 
-\frac{1}{\beta N} \ln \int_{\rens} \alpha_N\left(x,y^N\right) \, dy^N
\right]
\end{equation}
where $\dis D_N = \frac{1}{\beta N} \ln C_N$. Recall now that 
$$
\gamma_N(x) = \int_{\rens} \alpha_N\left(x,y^N\right) \, dy^N
$$
is the density of $\dis \frac{1}{N} \sum_{i=1}^N Y^i$, where $\left\{
  Y^i \right\}_{i=1}^N$ is a realization of the Markov chain of kernel
$g_0$. The behaviour of 
$\gamma_N$ when $N \to \infty$ is given by Assumption~\ref{assum:LDP_MC}:
\begin{equation}
\label{eq:lim2}
\lim_{N \to +\infty}
-\frac{1}{N} \ln \gamma_N(x) 
= \overline{F}_\infty(x),
\end{equation}
where $\overline{F}_\infty$ is 
given by~\eqref{eq:def_Finfty_NNN}. Collecting~\eqref{eq:lim1}
and~\eqref{eq:lim2}, we hence obtain that
$$
\lim_{N \to \infty} \left( F_N(x) + D_N \right) 
=
\frac{1}{\beta} \overline{F}_\infty(x).
$$
We thus have the following result:
\begin{lemma}
\label{lem:NNN_f}
Assume that $W_1$ and $W_2$ both belong to $L^1_{\rm
  loc}(\rens)$, that they are bounded from below, and that, for any $x
\in \rens$, we have $| W_1(x) | < \infty$ and $| W_2(x) | < \infty$.
In addition, we assume that
$e^{-\beta W_1}$ and $e^{-\beta W_2}$ both belong to $W^{1,1}_{\rm
  loc}(\rens)$, with
$$
\int_\rens e^{-\beta W_1(x)} dx < +\infty
\quad \text{and} \quad 
\int_\rens e^{-\beta W_2(x)}dx <+\infty,
$$
and that, for any $\xi \in \rens$, we have
$\exp \left( \xi x -\beta W_1(x) \right) \in L^1(\rens)$.

Under Assumption~\ref{assum:LDP_MC} for the kernel $g_0$ defined
by~\eqref{eq:def_g}, the limit of~\eqref{eq:FN_NNN} is 
given by 
\begin{equation}
\label{eq:LDP_NNN_bis}
\lim_{N \to +\infty}
\left( F_N(x) + C_N \right) 
= F_\infty(x)
\end{equation}
where $C_N$ is a constant that does not depend on $x$, and
$F_\infty$ is given by the Legendre transform
\begin{equation}
\label{eq:def_Finfty_NNN_bis}
F_\infty(x) := 
\frac{1}{\beta} \sup_{\xi \in \rens} \left( 
\xi x - \ln \Lambda(\xi) \right)
\end{equation}
where $\Lambda(\xi)$ 
is the largest eigenvalue of the operator (defined on $L^2(\rens)$)
\begin{equation} 
\label{eq:def_Qxi}
(Q_\xi \varphi)(y) = \int_\rens \varphi(x) \, \exp(\xi y) \, g_0(x,y) \, dx
\end{equation}
where $g_0$ is defined by~\eqref{eq:def_g}.
The convergence~\eqref{eq:LDP_NNN_bis} holds pointwise in $x$, and also in
$L^p_{\rm loc}$, for any $1 \leq p < \infty$.
As a consequence, the macroscopic force in the chain ${\cal T}_N(x) = F_N'(x)$
converges to $F'_\infty$ in $W^{-1,p}_{\rm loc}$. 
\end{lemma}

We hence obtain the macroscopic force $F_\infty'(x)$ for a prescribed
elongation $x$.
Note that, under our assumptions, in view of its
definition~\eqref{eq:def_Finfty_NNN_bis}, 
$F_\infty$ is (up to the factor $\beta$) the Legendre transform of some
function. It is hence always a convex function. Thus, as in the zero
temperature case, we 
observe, in this one-dimensional setting, that the macroscopic
constitutive law $x \mapsto F_\infty(x)$ is a convex function.

\begin{remark}
\label{rem:aussi}
In view of the definition~\eqref{eq:def_g} of $g_0$, we see that
$$
\frac{(Q_\xi \varphi)(y)}{\psi_0(y)} 
= 
\frac{1}{\lambda_0} 
\int_\rens \frac{\varphi(x)}{\psi_0(x)} \, \exp(\xi y) \, K_0(x,y) \, dx.
$$
Thus $\Lambda(\xi)$ is also the largest eigenvalue of the operator
$$
(\widetilde{Q}_\xi \varphi)(y)
= 
\frac{1}{\lambda_0} 
\int_\rens \varphi(x) \, \exp(\xi y) \, K_0(x,y) \, dx.
$$
Furthermore, if $\lambda$ is an eigenvalue of $\widetilde{Q}_\xi$,
then
$$
\int_\rens \varphi(x) \, \exp(\xi y) \, K_0(x,y) \, dx
= 
\lambda_0 \lambda \varphi(y)
$$
where $\varphi$ is an associated eigenfunction. Thus
$$
\int_\rens \frac{\varphi(x)}{\exp(\xi x/2)} \, \exp(\xi y/2) \, 
\exp(\xi x/2) \, K_0(x,y) \, dx
= 
\lambda_0 \lambda \frac{\varphi(y)}{\exp(\xi y/2)}
$$
and $\lambda_0 \lambda$ is an eigenvalue of the operator 
$$
(\overline{Q}_\xi \varphi)(y)
= 
\int_\rens \varphi(x) \, \exp(\xi y/2) \, \exp(\xi x/2) \, K_0(x,y) \, dx.
$$
The converse is also true. As $\Lambda(\xi)$ is the largest eigenvalue
of the operator $\widetilde{Q}_\xi$, we have that $\lambda_0
\Lambda(\xi)$ is the largest eigenvalue of the operator
$\overline{Q}_\xi$. 

Note that $\overline{Q}_\xi$ is a self-adjoint compact operator on
$L^2(\rens)$, which is thus easier to manipulate theoretically and
numerically than $Q_\xi$. In particular, using standard tools of spectral
theory of self-adjoint operators (see {\em e.g.}~\cite{dunford}), one
can prove that the largest eigenvalue of $\overline{Q}_\xi$ is simple,
and that the associated eigenvector $\Psi_\xi$ (which is unique up to a
multiplicative constant) can be chosen such that $\Psi_\xi>0$.
$\diamond$
\end{remark}

\subsubsection{Equivalence of stress-strain relations in the
  thermodynamic limit}

In Section~\ref{sec:NNN_y}, we have identified the function $f \mapsto
y^\star(f)$, that associates to a prescribed force $f$ the macroscopic
elongation $y^\star(f)$. Next, in Section~\ref{sec:NNN_f}, we have
identified the function $x \mapsto F'_\infty(x)$, that associates to a
prescribed elongation $x$ the macroscopic force $F'_\infty(x)$. We show
now that these functions are reciprocal one to each other.

Consider the optimization problem~\eqref{eq:def_Finfty_NNN_bis}. Since
the function $\xi \mapsto \ln \Lambda(\xi)$ is convex (see
Assumption~\ref{assum:LDP_MC}), there exists a unique maximizer
$\xi(x)$ in~\eqref{eq:def_Finfty_NNN_bis}, which satisfies the
Euler-Lagrange equation
\begin{equation}
\label{eq:rel1}
x = \frac{\Lambda'(\xi(x))}{\Lambda(\xi(x))}.
\end{equation} 
We also observe that
$$
F_\infty'(x) = \frac{\xi(x)}{\beta}.
$$

We see from~\eqref{eq:rel1} that we need to compute
$\Lambda'(\xi)$. Recall that $\Lambda(\xi)$ is the largest eigenvalue of the
operator~\eqref{eq:def_Qxi}. In view of Remark~\ref{rem:aussi},
$\lambda_0 \Lambda(\xi)$ is 
also the largest eigenvalue of $\overline{Q}_\xi$. Denoting $\Psi_\xi$
the associated eigenfunction satisfying $\| \Psi_\xi \|_{L^2} = 1$ and
$\Psi_\xi > 0$, we thus have 
$$
(\overline{Q}_\xi \Psi_\xi)(y) = 
\int_\rens \Psi_\xi(t) \, K_0^\xi(t,y) \, dt
=
\lambda_0 \Lambda(\xi) \Psi_\xi(y) 
$$
where 
\begin{equation}
\label{eq:def_K0xi}
K_0^\xi(t,y) = \exp(\xi y/2) \, \exp(\xi t/2) \, K_0(t,y).
\end{equation}
Multiplying by $\Psi_\xi(y)$ and integrating, we obtain
\begin{equation}
\label{eq:rel2}
\int_{\rens^2} \Psi_\xi(y) \, \Psi_\xi(t) \, K_0^\xi(t,y) \, dt
\, dy
=
\lambda_0 \Lambda(\xi).
\end{equation}
We thus have, using that $\overline{Q}_\xi$ is self-adjoint,
\begin{eqnarray*}
\lambda_0 \Lambda'(\xi)
&=&
\int_{\rens^2} \frac{d\Psi_\xi}{d\xi}(y) \, \Psi_\xi(t) \, 
K_0^\xi(t,y) \, dt \, dy
+
\int_{\rens^2} \Psi_\xi(y) \, \frac{d\Psi_\xi}{d\xi}(t) \, 
K_0^\xi(t,y) \, dt \, dy
\\
&& +
\int_{\rens^2} \Psi_\xi(y) \, \Psi_\xi(t) \, \frac{dK_0^\xi}{d\xi} (t,y) 
\, dt \, dy
\\
&=&
2 \lambda_0 \Lambda(\xi) \int_{\rens} \frac{d\Psi_\xi}{d\xi}(y) \, \Psi_\xi(y) 
\, dy
+
\int_{\rens^2} \Psi_\xi(y) \, \Psi_\xi(t) \, \frac{dK_0^\xi}{d\xi} (t,y) 
\, dt \, dy.
\end{eqnarray*}
In the above expression, the first term vanishes, since, for any $\xi$, 
$\dis \int_{\rens} \Psi^2_\xi(y) \, dy = 1$. We thus obtain
\begin{equation}
\label{eq:rel3}
\lambda_0 \Lambda'(\xi)
=
\int_{\rens^2} \Psi_\xi(y) \, \Psi_\xi(t) \, \frac{t+y}{2} \, K_0^\xi(t,y) 
\, dt \, dy.
\end{equation}
Collecting~\eqref{eq:rel1}, \eqref{eq:rel2} and~\eqref{eq:rel3}, we
see that
\begin{eqnarray}
\nonumber
x &=& \frac{\dis
\int_{\rens^2} \Psi_{\xi(x)}(y) \, \Psi_{\xi(x)}(t) \, \frac{t+y}{2} 
\, K_0^{\xi(x)}(t,y) 
\, dt \, dy
}{\dis
\int_{\rens^2} \Psi_{\xi(x)}(y) \, \Psi_{\xi(x)}(t) 
\, K_0^{\xi(x)}(t,y)
\, dt \, dy
}
\\
\nonumber
&=&
\frac{\dis
\int_{\rens^2} y \ \Psi_{\xi(x)}(y) \, \Psi_{\xi(x)}(t) 
\, K_0^{\xi(x)}(t,y) 
\, dt \, dy
}{\dis
\int_{\rens^2} \Psi_{\xi(x)}(y) \, \Psi_{\xi(x)}(t) 
\, K_0^{\xi(x)}(t,y)
\, dt \, dy
}
\\
\nonumber
&=&
\frac{\dis
\int_{\rens} y \ \Psi^2_{\xi(x)}(y) \, dy
}{\dis
\int_{\rens} \Psi^2_{\xi(x)}(y) \, dy
}
\\
\label{eq:resu1}
&=&
\int_{\rens} y \ \Psi^2_{\xi(x)}(y) \, dy
\end{eqnarray}
where we have used, at the second line, that 
$K_0^{\xi(x)}(t,y) = K_0^{\xi(x)}(y,t)$. 

On the other hand, we have obtained that the macroscopic elongation
$y^\star(f)$, for a prescribed force $f$, is given
by~\eqref{eq:def_ystar_NNN}, namely 
$$
y^\star(f) = 
\int_\rens y \ \psi_f^2(y) \, dy
$$
where $\psi_f$ is the maximizer of the variational
problem~\eqref{eq:psi1}. As $K_f$ is symmetric, the Euler-Lagrange
equation of~\eqref{eq:psi1} reads
\begin{eqnarray*}
\lambda_f \psi_f(y)
&=&
\int_{\rens} \psi_f(t) \ K_f(t,y) \ dt 
\\
&=& 
\int_{\rens} \psi_f(t) \ K_0(t,y) \exp \left(\beta f \frac{x+y}{2} \right)
\ dt 
\\
&=& 
\int_{\rens} \psi_f(t) \ K_0^{\beta f}(t,y) \, dt 
\end{eqnarray*}
where $K_0^{\beta f}$ is defined by~\eqref{eq:def_K0xi}.
Thus $\psi_f$ is an eigenfunction associated to the largest eigenvalue
$\lambda_f$ of the Hilbert-Schmidt operator $\overline{Q}_{\beta f}$ of
kernel $K_0^{\beta f}$. By definition of $\Psi_{\beta f}$, and using the
fact that the largest eigenvalue of $\overline{Q}_{\beta f}$ is simple, 
we obtain
$$
\Psi_{\beta f} = \pm \psi_f 
\quad \text{and} \quad
\Lambda(\beta f) = \frac{\lambda_f}{\lambda_0}.
$$
Thus
\begin{equation}
\label{eq:resu2}
y^\star(f) =
\int_\rens y \ \Psi_{\beta f}^2(y) \, dy.
\end{equation}
We deduce from the comparison of~\eqref{eq:resu1} and~\eqref{eq:resu2}
that $y^\star(\beta^{-1} \xi(x)) = y^\star(F'_\infty(x)) = x$.
Recall now that the function $f \mapsto y^\star(f)$ is increasing, as
shown by Lemma~\ref{lem:croissance}. It is thus injective, and we also
get the converse relation $F_\infty'(y^\star(f)) = f$. 

As a consequence, as in the NN setting considered in
Section~\ref{sec:equiv_NN}, the relation $f \mapsto y^\star(f)$ and $x
\mapsto F'_\infty(x)$ are inverse one to each
other. Prescribing a microscopic force $f$ and computing the macroscopic
elongation is equivalent to prescribing an elongation and computing the
macroscopic force, {\em in the thermodynamic limit}.

\subsubsection{Numerical computation of $F_\infty'$ and comparison with
  the zero temperature model}
\label{sec:NNN_num} 

For our numerical tests, we follow the choices made in~\cite{bllp}, for
the sake of comparison. We thus take the pair interaction potentials
$$
W_1(x) = \frac12 (x-1)^4 + \frac12 x^2
\quad \text{and} \quad
W_2(x) = \frac14 (x-2.1)^4.
$$
Note that these potentials satisfy all the assumptions that we have made
above. 

We are going to compare the free energy derivative ${\cal T}_N(x) = 
F_N'(x)$ with its
thermodynamic limit approximation $F_\infty'(x)$. The reference 
value $F_N'(x)$ is computed as the ensemble average
(\ref{eq:def_TN}), along the lines of~\eqref{eq:X}-\eqref{eq:ergo}. To
compute $F_\infty'(x)$, we proceed as follows: 
\begin{itemize}
\item[(i)] We first compute the largest eigenvalue $\Lambda(\xi)$ of the
  operator~\eqref{eq:def_Qxi}, for all $\xi$ in some prescribed interval.
\item[(ii)]~For any fixed $x$ in a prescribed interval, we next consider
  the variational problem~\eqref{eq:def_Finfty_NNN_bis}, compute its
  maximizer $\xi(x)$, and obtain $F_\infty'(x)$ using 
$\dis F_\infty'(x) = \frac{\xi(x)}{\beta}$.
\end{itemize}  
In practice, using Remark~\ref{rem:aussi}, we work with the operator
$\overline{Q}_\xi$, which is easier to manipulate since it is
self-adjoint and we do not need to first solve~\eqref{eq:psi1}. We thus
first compute the largest eigenvalue $\lambda_0 \Lambda(\xi)$ of 
$\overline{Q}_\xi$, and next compute the Legendre
transform of the function $\xi \mapsto \ln (\lambda_0
\Lambda(\xi))$. The maximizer is the same as that for $F_\infty(x)$.
On Fig.~\ref{fig:ln_lambda_xi}, we plot the function $\xi \mapsto \ln
\left( \lambda_0 \Lambda(\xi) \right)$, and observe that it is
convex, in agreement with Assumption~\ref{assum:LDP_MC}.

\begin{figure}[htbp]
\centerline{
\input{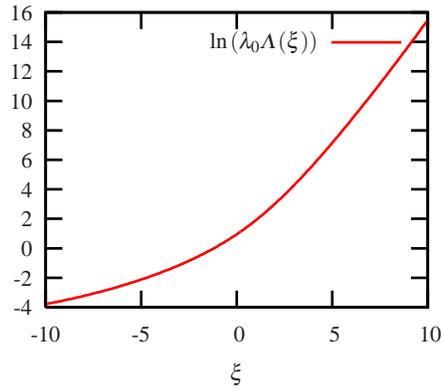}
}
\caption{Plot of $\ln \left( \lambda_0 \Lambda(\xi) \right)$ as a
  function of $\xi$ (temperature $1/\beta=1$).   
}
\label{fig:ln_lambda_xi}
\end{figure}

\medskip

We first study the convergence of
$F_N'(x)$ to $F'_\infty(x)$ as $N$ increases, for a fixed chain
length $x = 1.4$ and a fixed temperature $1/\beta = 1$. Results are
shown on Figure~\ref{fig:conv_N_force_macro}. We indeed observe that
$F_N'(x) \to F'_\infty(x)$ when $N \to +\infty$.

\begin{figure}[htbp]
\centerline{
\input{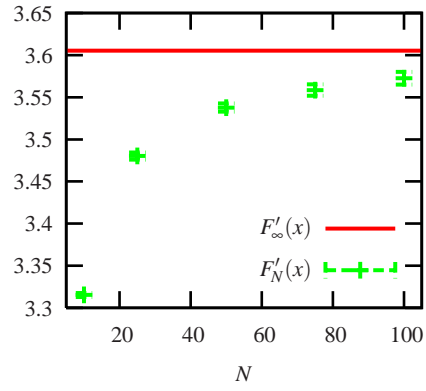}
}
\caption{Convergence of $F_N'(x)$ (shown with error bars computed from
  40 independent realizations) to
  $F_\infty'(x)$ as $N$ increases (temperature $1/\beta=1$, fixed chain
  length $x=1.4$).  
}
\label{fig:conv_N_force_macro}
\end{figure}

We now compare $F'_N(x)$ with its approximation $F'_\infty(x)$,
for $N=100$ and $1/\beta=1$. Results are shown on
Figure~\ref{fig:macro_force_n100}. We observe that $F'_\infty(x)$ is a
very good approximation of $F'_N(x)$, for any $x$ in the considered
interval.

\medskip

For the sake of comparison, we now identify the zero temperature
behaviour of the system, in the thermodynamic limit. At zero
temperature, for a finite $N$, we model the system by minimizing the
energy $E_0$, with prescribed Dirichlet boundary conditions (this corresponds
to prescribing the elongation, and computing the force; alternatively,
one could impose Neumann boundary conditions, {\em i.e.} prescribe a
force and compute an elongation):
\begin{equation}
\label{eq:pb_var}
J_N(x) = \frac1N
\inf \left\{ E_0\left(u^0,u^1,\ldots,u^{N-1},u^N\right), \ 
u^0 = 0, \ u^N = x \right\}.
\end{equation}
We have the following result, which proof will be given below:
\begin{lemma}
\label{lem:zero_T}
Let us introduce $\phi$ defined by
\begin{equation}
\label{eq:def_phi}
\phi(x) = W_1(x) + W_2(2 x).
\end{equation}
Assume that there exists $\alpha>0$ such that
\begin{equation}
\label{eq:hypo_w1}
W_1(x) \geq \alpha x^2,
\end{equation}
and that $W_1$ and $\phi$ are non-negative and strictly convex
functions. Then we have the pointwise convergence 
$$
\lim_{N \to \infty} J_N(x) = \phi(x).
$$

Assume in addition that $\phi \in L^p_{\rm loc}$ for
some $1 \leq p < \infty$ and that $W_2$ is non-negative. Then the above
convergence also holds in $L^p_{\rm loc}$. As a consequence, $J_N'(x)$
converges to $\phi'(x)$ in $W^{-1,p}_{\rm loc}$.  
\end{lemma}

When the temperature is set to zero, the energy thus converges, in the
thermodynamic limit, to $\phi(x)$, and the force ({\em i.e.} the derivative of
the energy with respect to the prescribed Dirichlet boundary condition)
converges to $\phi'(x)$. We plot on Figure~\ref{fig:macro_force_n100}
the function $x \mapsto \phi'(x)$. 
We clearly observe the effect of temperature, as
$F'_\infty(x)$ for $\beta = 1$ significantly differs from $\phi'(x)$. 

\begin{figure}[htbp]
\centerline{
\input{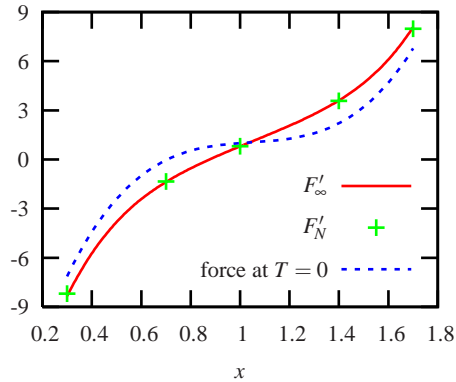}
}
\caption{We plot $F_N'(x)$ and $F_\infty'(x)$ 
  for the temperature $1/\beta=1$ and $N=100$. On the scale of the
  figure, $F_N'(x)$ and $F_\infty'(x)$ are on top of each other. We also
plot the zero temperature response $\phi'(x)$.}
\label{fig:macro_force_n100}
\end{figure}

\begin{proof}[Lemma~\ref{lem:zero_T}]
\smartqed
Let
$$
X_N(x) = \left\{ \left(u^0,u^1,\dots,u^{N-1},u^N\right) \in \rens^{1+N}, 
\ u^0 = 0, \ u^N = x \right\}
$$
be the variational ensemble for the problem~\eqref{eq:pb_var}. The
configuration $u^i = i x/N$ clearly belongs to that ensemble. We thus
obtain the upper-bound
\begin{equation}
\label{eq:upper_bound}
J_N(x) \leq W_1(x) 
+
\frac{N-1}N W_2(2x).
\end{equation}
In the sequel, we first show a lower-bound for $J_N(x)$, and next study
its behaviour when $N \to \infty$. 

Let us first build a lower bound for $J_N(x)$. Assuming for the sake of
simplicity that $N$ is even, and using the short-hand notation
$\dis y^i = \frac{u^{i} - u^{i-1}}{h}$, we have
\begin{eqnarray*}
\frac1N \sum_{i=1}^{N} W_1\left( \frac{u^{i} - u^{i-1}}{h} \right)
&=&
\frac1N \sum_{i=1}^{N} W_1\left( y^i \right)
\\
&=&
\frac{1}{2N} W_1(y^1) + \frac{1}{2N} W_1(y^N)
+
\frac{1}{2N} \sum_{i=1}^{N/2} 
\left[ W_1\left( y^{2i-1} \right) + W_1\left( y^{2i} \right) \right]
\\
&& +
\frac{1}{2N} \sum_{i=1}^{N/2-1} 
\left[ W_1\left( y^{2i} \right) + W_1\left( y^{2i+1} \right) \right].
\end{eqnarray*}
By convexity of $W_1$, we obtain
\begin{eqnarray*}
\frac1N \sum_{i=1}^{N} W_1\left( \frac{u^{i} - u^{i-1}}{h} \right)
&\geq&
\frac{1}{2N} W_1(y^1) + \frac{1}{2N} W_1(y^N)
+
\frac{1}{N} \sum_{i=1}^{N/2} W_1  
\left[ \frac12 \left( y^{2i-1} + y^{2i} \right) \right]
\\
&& +
\frac{1}{N} \sum_{i=1}^{N/2-1} W_1 
\left[ \frac12 \left( y^{2i} + y^{2i+1} \right) \right].
\end{eqnarray*}
Taking into account the next-to-nearest interactions, we thus obtain
that, for any $\left(u^0,u^1,\dots,u^{N-1},u^N\right) \in \rens^{1+N}$,
\begin{eqnarray*}
\frac1N E_0\left(u^0,u^1,\ldots,u^{N-1},u^N\right)
&\geq&
\frac{1}{2N} W_1(y^1) + \frac{1}{2N} W_1(y^N)
+
\frac{1}{N} \sum_{i=1}^{N/2} \phi  
\left[ \frac12 \left( y^{2i-1} + y^{2i} \right) \right]
\\
&& +
\frac{1}{N} \sum_{i=1}^{N/2-1} \phi
\left[ \frac12 \left( y^{2i} + y^{2i+1} \right) \right],
\end{eqnarray*}
where $\phi$ is defined by~\eqref{eq:def_phi}. As $\phi$ is convex, we
deduce that
\begin{eqnarray*}
\frac1N E_0\left( u \right)
&\geq&
\frac{1}{2N} W_1(y^1) + \frac{1}{2N} W_1(y^N)
+
\frac{1}{2} \phi \left( \frac{1}{N} \sum_{i=1}^{N/2} 
\left[ y^{2i-1} + y^{2i} \right] \right)
\\
&& +
\frac{N-2}{2N} \phi \left( \frac{1}{N-2} \sum_{i=1}^{N/2-1}  
\left[ y^{2i} + y^{2i+1} \right] \right)
\\
&=&
\frac{1}{2N} W_1(y^1) + \frac{1}{2N} W_1(y^N)
+
\frac{1}{2} \phi \left( u^{N} - u^{0} \right)
\\
&& +
\frac{N-2}{2N} \phi \left( \frac{N}{N-2}   
\left( u^{N-1} - u^{1} \right) \right).
\end{eqnarray*}
As a consequence, for any configuration $u \in X_N(x)$, we have
\begin{equation}
\label{eq:tata1}
\frac1N E_0\left(u^0,u^1,\ldots,u^{N-1},u^N\right)
\geq
\overline{E}_N(u^1,u^{N-1};x)
\end{equation}
with
\begin{eqnarray*}
\overline{E}_N(u^1,u^{N-1};x)
&=&
\frac{1}{2N} W_1(N u^1) + \frac{1}{2N} W_1(N(x-u^{N-1}))
+
\frac{1}{2} \phi(x)
\\
&& +
\frac{N-2}{2N} \phi \left( \frac{N}{N-2}   
\left( u^{N-1} - u^{1} \right) \right).
\end{eqnarray*}
We infer from~\eqref{eq:tata1} the lower bound
\begin{equation}
\label{eq:llower_bound}
J_N(x) 
\geq
\overline{J}_N(x)
\end{equation}
with
\begin{equation}
\label{eq:pb_var_aux}
\overline{J}_N(x) =
\inf \left\{ 
\overline{E}_N(u^1,u^{N-1};x); \ u^1 \in \rens, \ u^{N-1} \in \rens
\right\}.
\end{equation}

We now study the auxiliary variational problem~\eqref{eq:pb_var_aux} to
determine the limit of $\overline{J}_N(x)$ when $N \to \infty$. 
Since $\phi$ is non-negative, we infer from~\eqref{eq:hypo_w1} that 
$$
\overline{E}_N(u^1,u^{N-1};x)
\geq
\frac{\alpha N}{2} \left[ (u^1)^2 + (x-u^{N-1})^2 \right] \geq 0.
$$
As a consequence, $\overline{J}_N(x) \geq 0$, and any minimizing
sequence is bounded. Up to extraction, it thus converges to a minimizer,
that we denote $\left(\overline{u}^1,\overline{u}^{N-1}\right)$.
As $W_1$ and $\phi$ are strictly
convex, it is easy to see that the hessian matrix of $\overline{E}_N$ is
positive definite, hence $\overline{E}_N$ is also strictly convex, hence
it has a unique minimizer.  
The problem~\eqref{eq:pb_var_aux} is thus well-posed. To underline the
dependency of its minimizer with $N$, we denote it 
$\left(\overline{u}^1(N),\overline{u}^{N-1}(N)\right)$ in the sequel.

The Euler-Lagrange equation associated to~\eqref{eq:pb_var_aux} reads
$$
W_1'(N \overline{u}^1(N)) 
= 
\phi' \left( \frac{N}{N-2}   
\left( \overline{u}^{N-1}(N) - \overline{u}^{1}(N) \right) \right)
=
W_1'(N(x-\overline{u}^{N-1}(N))).
$$
As $W_1$ is strictly convex, this implies that
\begin{equation}
\label{eq:EL}
\left\{
\begin{array}{rcl}
\overline{u}^1(N) &=& x-\overline{u}^{N-1}(N),
\\ \noalign{\vskip 5pt}
N \overline{u}^1(N) &=& 
\dis \chi \left( \frac{N}{N-2}   
\left( x - 2 \overline{u}^{1}(N) \right) \right),
\end{array}
\right.
\end{equation}
where the function $\chi = (W_1')^{-1} \circ \phi'$ is independent of
$N$, and increasing.

Let us now show that $\overline{u}^1(N)$ is bounded with respect to $N$. If
this is not the case, then, without loss of generality, it is possible
to find a subsequence $\varphi(N)$ such that 
$\lim_{N \to \infty} \overline{u}^1(\varphi(N)) = +\infty$. Passing to
the limit in the second line of~\eqref{eq:EL}, one obtains a
contradiction. Thus $\overline{u}^1(N)$ is bounded. 

In view of the first line of~\eqref{eq:EL}, $\overline{u}^{N-1}(N)$ is also
bounded. Up to a subsequence extraction,
$(\overline{u}^1(N),\overline{u}^{N-1}(N))$ converges when $N \to
\infty$ to $(\overline{u}^1(\infty),\overline{u}^{N-1}(\infty))$. We
infer from~\eqref{eq:EL} that $\overline{u}^1(\infty) = 0$ and
$\overline{u}^{N-1}(\infty) = x$, thus the limit is unique, and all the
sequence converges:
\begin{equation}
\label{eq:lim1_}
\lim_{N \to \infty} \overline{u}^1(N) = 0,
\quad
\lim_{N \to \infty} \overline{u}^{N-1}(N) = x.
\end{equation}
We next infer from the above limits and~\eqref{eq:EL} that 
\begin{equation}
\label{eq:lim2_}
\lim_{N \to \infty} N \overline{u}^1(N) 
=
\lim_{N \to \infty} 
N(x-\overline{u}^{N-1}(N))
= \chi(x).
\end{equation}
By definition, we have
$$
\overline{J}_N(x)
= 
\inf \left\{ 
\overline{E}_N(u^1,u^{N-1};x); \ u^1 \in \rens, \ u^{N-1} \in \rens
\right\}
=
\overline{E}_N(\overline{u}^1(N), \overline{u}^{N-1}(N);x).
$$
In view of~\eqref{eq:lim1_} and~\eqref{eq:lim2_}, we obtain
\begin{equation}
\label{eq:tata2}
\lim_{N \to \infty} \overline{J}_N(x)
= 
\lim_{N \to \infty} \overline{E}_N(\overline{u}^1(N), \overline{u}^{N-1}(N);x)
=
\phi(x).
\end{equation}
Collecting~\eqref{eq:upper_bound}, \eqref{eq:llower_bound}
and~\eqref{eq:tata2}, we obtain the claimed pointwise convergence of
$J_N(x)$ to $\phi(x)$.

\medskip

Under the additional assumption that $W_2$ is non-negative, we deduce
from~\eqref{eq:upper_bound} that, for any $N$ and any $x$,
$$
0 \leq J_N(x) \leq W_1(x) + W_2(2x) = \phi(x).
$$
As $\phi \in L^p_{\rm loc}$, we obtain the convergence
of $J_N$ to $\phi$ in $L^p_{\rm loc}$.
\qed
\end{proof}

\section{A coarse-graining procedure in the dynamical setting}
\label{sec:eff_dyn} 

In this section, we present a procedure for coarse-graining a
dynamics. More precisely, we consider $Q_t \in \rens^n$ solution to the
overdamped dynamics~\eqref{eq:X}, and a reaction coordinate $\xi :
\rens^n \mapsto \rens$. Our aim is to find a closed one-dimensional
dynamics of type~\eqref{eq:eff_dyn} on a process $\eeta_t$, such that
$\eeta_t$ is a good approximation of $\xi(Q_t)$. In Sections~\ref{sec:coeur}
and~\ref{sec:pedago}, we build such a process (see~\eqref{eq:y} below),
and present an analytical estimation of its accuracy (the obtained
estimate is an upper-bound on the ``distance'' between the laws of
$\xi(Q_t)$ and $\eeta_t$ 
at any time $t$). We will next report on some  
numerical experiments that somewhat check the accuracy of $\eeta_t$ in
a stronger way (Section~\ref{sec:eff_dyn_num}).

\subsection{Measuring distances between probability measures}

We introduce here some tools that will be useful in the sequel, to
measure how close two probability measures are.  
Consider two probability measures $\nu(dq)$ and $\eta(dq)$. The distance
between the two can be measured by the total variation norm
$\| \nu - \eta \|_{\rm TV}$, which amounts to the $L^1$-norm 
$\dis \int \left|\psi_\nu(q) - \psi_\eta(q)\right| \, dq$ in case $\nu$
and $\eta$ have respectively the densities $\psi_\nu$ and $\psi_\eta$
with respect to the Lebesgue measure.

When studying the long-time behaviour of solutions to PDEs (such as long
time convergence of the solution of a Fokker-Planck equation to the
stationary measure of the corresponding SDE), the notion of relative
entropy turns out to be more useful. Under the assumption that $\nu$ is absolutely continuous with respect to $\eta$ (denoted $\nu \ll \eta$ in the sequel), it is
defined by
$$
H \left( \nu | \eta \right) = 
\int \ln\left(\frac{d\nu}{d\eta}\right) d\nu.
$$
The relative entropy provides an upper-bound on the total variation norm, by
the Csisz\'ar-Kullback inequality: 
$$
\| \nu - \eta \|_{\rm TV} \leq \sqrt{2 H\left( \nu | \eta \right)}.
$$
In the sequel, we will also use the Wasserstein distance with quadratic cost, which
is defined, for any two probability measures $\nu$ and $\eta$ with
support on a 
Riemannian manifold $\Sigma$, by 
\begin{equation}
\label{eq:wasser} 
W(\nu,\eta) = \sqrt{\inf_{\pi \in \Pi(\nu,\eta)} 
\int_{\Sigma \times \Sigma} d_{\Sigma}(x,y)^2 \  \pi(dx,dy)}.
\end{equation}
In the above expression, $d_{\Sigma}(x,y)$ denotes the geodesic distance
between $x$ and $y$ on $\Sigma$,
$$
d_{\Sigma}(x,y) = \inf \left\{ \sqrt{ \int_0^1 \left| \dot{\alpha}(t) \right|^2 \, dt};
\ \alpha \in C^1([0,1],\Sigma), \, \alpha(0) = x, \, \alpha(1) = y \right\},
$$
and $\Pi(\nu,\eta)$ denotes the set of coupling probability measures, that is
probability measures $\pi$ on $\Sigma \times \Sigma$ such that
their marginals are $\nu$ and $\eta$: for any test function $\Phi$, 
$$
\int_{\Sigma \times \Sigma} \Phi(x) \, \pi(dx,dy) = \int_{\Sigma} \Phi(x)
\, \nu(dx)
\ \ \text{and} \ \
\int_{\Sigma \times \Sigma} \Phi(y) \, \pi(dx,dy) = \int_{\Sigma}
  \Phi(y) \, \eta(dy).
$$

In the sequel, we will need two functional inequalities, that we now
recall~\cite{ABC-00}:

\begin{definition}
A probability measure $\eta$ satisfies a logarithmic
Sobolev inequality with a constant $\rho>0$ if, for any probability
measure $\nu$ such that $\nu \ll \eta$, 
$$
H(\nu | \eta) \leq \frac{1}{2 \rho} I(\nu | \eta)
$$
where the Fisher information $I(\nu | \eta)$ is defined by
$$
I(\nu | \eta) = \int \left| \nabla \ln \left( \frac{d\nu}{d\eta} \right)
\right|^2 d\nu. 
$$
\end{definition}

\begin{definition}
A probability measure $\eta$ satisfies a Talagrand inequality with a
constant $\rho>0$ if, for any probability measure $\nu$,
$$
W(\nu,\eta) \leq \sqrt{ \frac{2}{\rho} H(\nu | \eta)}.
$$
\end{definition}

We will also need the following important result
(see~\cite[Theorem~1]{otto-villani-00} and~\cite{bobkov}): 
\begin{lemma}
\label{lem:lsi-t}
If $\eta$ satisfies a logarithmic Sobolev inequality with a constant
$\rho>0$, then $\eta$ satisfies a Talagrand inequality with the same
constant $\rho>0$.
\end{lemma}

The following standard result illustrates the usefulness of 
logarithmic Sobolev inequalities (we refer to
\cite{ABC-00,arnold-markowich-toscani-unterreiter-01,villani-03} for 
more details on this subject).
\begin{theorem}
Consider $Q_t$ solution to the overdamped Langevin
equation~\eqref{eq:X}, and assume the 
stationary measure $\psi_\infty(q) \, dq = Z^{-1} \exp(-\beta V(q)) \, dq$
satisfies a logarithmic Sobolev inequality with a constant
$\rho>0$. Then the probability distribution $\psi(t,\cdot)$ of $Q_t$
converges to $\psi_\infty$ exponentially fast, in the sense:
\begin{equation}\label{eq:ent_cv}
\forall t \ge 0, \quad H(\psi(t,\cdot) | \psi_\infty)
\leq 
H(\psi(0,\cdot) | \psi_\infty) \exp(-2 \rho \beta^{-1}t).
\end{equation}
Conversely, if~\eqref{eq:ent_cv} holds for any initial condition
$\psi(0,\cdot)$, then the stationary measure $\psi_\infty(q) \, dq$ 
satisfies a logarithmic Sobolev inequality with a constant
$\rho>0$.
\end{theorem}

\begin{proof}
\smartqed
The probability distribution function $\psi(t,q)$ of $Q_t$ satisfies the
Fokker-Planck equation 
\begin{equation}
\label{eq:FP_complet}
\partial_t \psi =
\dive \left( \psi \nabla V \right) + \beta^{-1} \Delta \psi.
\end{equation}
As $\nabla \psi_\infty = -\beta \psi_\infty \nabla V$, we recast the
above equation as 
$$
\partial_t \psi = 
\beta^{-1} \dive \left[ \psi_\infty 
\nabla \left( \frac{\psi}{\psi_\infty} \right) \right].
$$
Note that this equation implies that $\dis \int_{\rens^n} \psi(t,q) \, dq$ is a
constant. 
Introduce now the relative entropy 
$$
{\mathcal E}(t) = H(\psi(t,\cdot) | \psi_\infty) = 
\int_{\rens^n} \ln \left(\frac{\psi(t,q)}{\psi_\infty(q)} \right)
\psi(t,q) \, dq.
$$ 
Then
\begin{eqnarray}
\nonumber
\frac{d{\mathcal E}}{dt} &=& \int_{\rens^n} \ln \left(\frac{\psi}{\psi_\infty} \right)
\partial_t \psi + 
\frac{\psi_\infty}{\psi} \ \frac{\partial_t \psi}{\psi_\infty} \ \psi 
\\
\nonumber
&=&
\int_{\rens^n} \ln \left(\frac{\psi}{\psi_\infty} \right) 
\beta^{-1} \dive \left[ \psi_\infty 
\nabla \left( \frac{\psi}{\psi_\infty} \right) \right]
\\
\nonumber
&=&
- \beta^{-1} 
\int_{\rens^n} \nabla \left[ \ln \left(\frac{\psi}{\psi_\infty} \right) \right]
\psi_\infty 
\nabla \left( \frac{\psi}{\psi_\infty} \right)
\\
\nonumber
&=&
- \beta^{-1} 
\int_{\rens^n} \left| 
\nabla \left[ \ln \left(\frac{\psi}{\psi_\infty} \right) \right]
\right|^2
\psi
\\
\label{eq:ici}
&=& - \beta^{-1} I(\psi(t,\cdot) | \psi_\infty).
\end{eqnarray}
As $\psi_\infty$ satisfies a logarithmic Sobolev inequality with the constant
$\rho>0$, we have that, for any time $t \ge 0$, 
\begin{equation}
\label{eq:ici2}
H(\psi(t,\cdot) | \psi_\infty) \leq (2 \rho)^{-1} I(\psi(t,\cdot) | \psi_\infty).
\end{equation}
We infer from~\eqref{eq:ici} and~\eqref{eq:ici2} that
$$
\frac{d{\mathcal E}}{dt} \leq -2 \rho \beta^{-1} {\mathcal E}.
$$
Using the Gronwall lemma, we obtain the claimed result. 

Conversely, if 
$$
\forall t \ge 0, \quad {\mathcal E}(t) \le  {\mathcal
  E}(0) \exp(-2 \rho \beta^{-1} t),
$$ 
we also have
$$
\forall t > 0, \quad 
\frac{{\mathcal E}(t)-{\mathcal E}(0)}{t} \le  {\mathcal E}(0)
\frac{\exp(-2 \rho \beta^{-1} t) - 1}{t},
$$
and by letting $t$ go to $0$, using~\eqref{eq:ici}, one obtains the
logarithmic Sobolev inequality 
$I(\psi(0,\cdot) | \psi_\infty) \ge 2 \rho H(\psi(0,\cdot) |
\psi_\infty)$.
\qed
\end{proof}

\subsection{Effective dynamics}
\label{sec:coeur}

Consider $Q_t$ that solves~\eqref{eq:X}. 
By a simple It\^o computation, we have
\begin{equation}
\label{eq:xi(X)}
d \xi(Q_t) = \left(- \nabla V \cdot \nabla \xi + \beta^{-1} \Delta \xi
\right) (Q_t) \, dt + \sqrt{2 \beta^{-1}} \ |\nabla \xi|(Q_t) \, dB_t,
\end{equation}
where $B_t$ is the one-dimensional Brownian motion
$$
dB_t = \frac{\nabla\xi}{|\nabla\xi|}(Q_t) \cdot dW_t.
$$
Of course, equation~\eqref{eq:xi(X)} is not closed. Following
Gy\"ongy~\cite{gyongy-86}, a simple closing 
procedure is to consider $\teta_t$ solution to
\begin{equation}
\label{eq:ty}
d \teta_t = \tb(t, \teta_t) \, dt + \sqrt{2 \beta^{-1}} \
\tsigma(t, \teta_t) \, dB_t,
\end{equation}
where
\begin{eqnarray}
\label{eq:tb}
\tb(t,z)&=&\EE\left[ \left( - \nabla V \cdot \nabla \xi + \beta^{-1} \Delta \xi
\right) (Q_t) \ | \ \xi(Q_t)=z \right],
\\
\label{eq:tsigma}
\tsigma^2(t,z) &=& \EE\left[|\nabla \xi|^2(Q_t) \ | \ \xi(Q_t)=z \right].
\end{eqnarray}
Note that $\tb$ and $\tsigma$ depend on $t$, since these are expected
values conditioned on the fact that $\xi(Q_t)=z$, and the 
probability distribution function of $Q_t$ of course depends on $t$. 

As shown in~\cite{gyongy-86}, 
this procedure is exact from the point of view of time marginals: at
any time $t$, the random variables $\teta_t$ and $\xi(Q_t)$ have the
same law. 
This is stated in the following lemma.

\begin{lemma}[\cite{dyn_eff}, Lemma 2.3]
\label{lem:gyongy}
The probability distribution function $\psi^\xi$ of $\xi(Q_t)$, where
$Q_t$ satisfies~\eqref{eq:X}, satisfies 
the Fokker-Planck equation associated to~\eqref{eq:ty}: 
$$
\partial_t \psi^\xi = \partial_z\left( - \tb \ \psi^\xi + \beta^{-1}
  \partial_z (\tsigma^2 \psi^\xi )\right).
$$
\end{lemma}

The problem with equation~\eqref{eq:ty} is that the functions $\tb$
and $\tsigma$ are very complicated to compute, since they involve
the full knowledge of $\psi$. Therefore, one cannot
consider~\eqref{eq:ty} as a reasonable closure. A natural
simplification is to
consider a time-independent approximation of the functions $\tb$
and $\tsigma$. Considering~\eqref{eq:tb} and~\eqref{eq:tsigma}, we
introduce ($\EE_{\mu}$ denoting a mean with respect to the measure $\mu$)
\begin{eqnarray}
b(z)
&=&\EE_{\mu} \left[ \left(- \nabla V \cdot \nabla \xi + \beta^{-1} \Delta \xi
\right) (Q) \ | \ \xi(Q)=z \right] \nonumber \\
&=& \int_{\Sigma_z} \left(- \nabla V \cdot \nabla \xi + \beta^{-1} \Delta \xi
\right) d \mu_{\Sigma_z},\label{eq:b}
\end{eqnarray}
and
$$
\sigma^2(z)
=\EE_{\mu}\left(|\nabla \xi|^2(Q) \ | \ \xi(Q)=z \right)
=\int_{\Sigma_z} |\nabla \xi|^2 \ d \mu_{\Sigma_z},
$$
where $\mu_{\Sigma_z}$ is defined by~\eqref{eq:mu_z}. This
simplification especially makes sense if $\xi(Q_t)$ is a slow variable,
that is if the characteristic evolution time of $\xi(Q_t)$ is much
larger than the characteristic time needed by $Q_t$ to sample the
manifold $\Sigma_z$. This is quantified in the sequel. 

In the spirit of \eqref{eq:ty}, we
next introduce the coarse-grained dynamics
\begin{equation}
\label{eq:y}
d\eeta_t = b(\eeta_t) \, dt + \sqrt{2 \beta^{-1}} \ \sigma(\eeta_t) \, dB_t, 
\quad \eeta_{t=0}=\xi(Q_0).
\end{equation}
We have proved in~\cite{dyn_eff} that the effective dynamics~\eqref{eq:y} is
ergodic for the equilibrium measure $\xi \star \mu$, that is
$\exp(-\beta A(z)) \, dz$. In addition, this measure satisfies a
detailed balance condition. We have also proved the following error
bound, that quantifies the ``distance'' between the probability
distribution function of $\xi(Q_t)$ (at any given time $t$) and that of
$\eeta_t$. 

\begin{proposition}[\cite{dyn_eff}, Proposition 3.1]
\label{prop:D3}
Assume that $\xi$ is a smooth scalar function such that
\begin{equation}
 \label{H1}
% pas appele, numerote quand meme, c'est une hypothese
 \begin{array}{c}
 \text{for all } q \in \R^n, \quad 0 < m \le |\nabla \xi(q)| \le M < \infty,
 \end{array}
 \end{equation}
and 
that the conditioned probability measures $\mu_{\Sigma_z}$,
defined by~\eqref{eq:mu_z}, satisfy a logarithmic Sobolev inequality
with a constant $\rho$ uniform in $z$: for any probability measure $\nu$
on $\Sigma_z$ which is absolutely continuous with respect to the measure
$\mu_{\Sigma_z}$, we have 
\begin{equation}
\label{H2}
H(\nu|\mu_{\Sigma_z}) \le \frac{1}{2 \rho} I(\nu|\mu_{\Sigma_z}) .
\end{equation}
Let us also assume that the coupling is bounded in the following sense:
\begin{equation}
\label{H3}
\kappa = \|\nabla_{\Sigma_z} F \|_{L^\infty} < \infty,
\end{equation}
where $F$ is the local mean force defined by~\eqref{eq:F}. 

Finally, let us assume that $|\nabla \xi|$ is close to a constant on the
manifold $\Sigma_z$ in the following sense: 
\begin{equation}
\label{H4}
\lambda=\left\| \frac{|\nabla \xi|^2 - \sigma^2 \circ \xi}{\sigma^2 \circ \xi}
\right\|_{L^\infty} < \infty.
\end{equation}
Assume that, at time $t=0$, the distribution of the initial conditions
of~\eqref{eq:X} and~\eqref{eq:y} are consistent one with each other:
$\psi^\xi(t=0,\cdot) = \phi(t=0,\cdot)$. 
Then we have the following estimate: for any time $t \ge 0$,
\begin{equation}
\label{eq:D3estim}
E(t)
\le \frac{M^2}{4m^2} \left( \lambda^2 + \frac{m^2 \beta^2 \kappa^2}{\rho^2}
\right) \left(
H(\psi(0,\cdot) | \mu) - H(\psi(t,\cdot) | \mu)\right),
\end{equation}
where $E(t)$ is the relative entropy of the probability distribution
function $\psi^\xi$ of $\xi(Q_t)$, where $Q_t$ follows~\eqref{eq:X},
with respect to the probability distribution function $\phi$ of the solution
$\eeta_t$ to~\eqref{eq:y}:
$$
E(t)=H \left( \psi^\xi(t,\cdot) | \phi(t,\cdot) \right) = 
\int_\rens \ln\left(\frac{\psi^\xi(t,z)}{\phi(t,z)}\right) \psi^\xi(t,z) \,
dz.
$$
\end{proposition}

The above proposition thus yields a uniform-in-time bound on the
relative entropy between $\psi^\xi$ and $\phi$. In addition, we also know
that the effective dynamics is ergodic for $\exp(-\beta A(z)) \, dz$,
which is the equilibrium measure of $\xi(Q_t)$, in the long-time
limit. We thus expect the two probability densities to converge 
one to each other, in the long-time limit. This is indeed the case, as
it is shown in~\cite[Corollary 3.1]{dyn_eff}: under some mild
assumptions, the $L^1$ distance between $\psi^\xi(t,\cdot)$ and
$\phi(t,\cdot)$ vanishes at an exponential rate in the long-time limit. 

\subsection{The proof in a simple two-dimensional case}
\label{sec:pedago}

For the purpose of illustration, we consider in this section an extremely
simple case: starting from the overdamped dynamics~\eqref{eq:X} in {\em
  two dimensions} (we write $q=(x,y) \in \rens^2$), we want to derive an
effective 
dynamics for the coarse-grained variable $\xi(q) = \xi(x,y) = x$. Although
this case is over-simplified, it turns out that the main arguments of
our derivation, as well as the proof arguments, can be well understood here. 

\medskip

In that context, the complete dynamics~\eqref{eq:X} reads
\begin{equation}
\label{eq:X_2d}
\left\{
\begin{array}{rcl}
dX_t &=& \dis
- \partial_x V(X_t,Y_t) \, dt + \sqrt{2
  \beta^{-1}} \, d W^x_t,
\\ \noalign{\vskip 5pt}
dY_t &=& \dis
- \partial_y V(X_t,Y_t) \, dt + \sqrt{2
  \beta^{-1}} \, d W^y_t,
\end{array}
\right.
\end{equation}
with the initial condition $Q_0 = (X_0,Y_0)$. The manifold $\Sigma_z$
defined by~\eqref{eq:mani} is
$$
\Sigma_z = \left\{ (z,y); \ y \in \rens \right\}
$$
and the probability measure $d\mu_{\Sigma_z}$ defined by~\eqref{eq:mu_z}
reads
\begin{equation}
\label{eq:mu_z_2d}
d\mu_{\Sigma_z} = 
\frac{\exp(-\beta V(z,y)) dy}
{\dis \int_\rens \exp(-\beta V(z,y)) dy}
=
\frac{\psi_\infty(z,y) dy}
{\dis \int_\rens \psi_\infty(z,y) dy}.
\end{equation}

We focus on the dynamics of $\xi(X_t,Y_t) = X_t$.
In that case, the equation~\eqref{eq:xi(X)} is just the first line
of~\eqref{eq:X_2d}, 
which is obviously not closed in $X_t$, since $Y_t$ appears. At time $t$, $Q_t$ is
distributed according to the measure $\psi(t,q)$. Hence, the probability
distribution function of $Y_t$, conditioned to the fact that $\xi(Q_t) =
X_t = x$, is given by
$$
\psi_{\rm cond}^x(t,y) =
\frac{\psi(t,x,y)}{\dis \int_\rens \psi(t,x,y) \, dy}.
$$
Following Gy\"ongy~\cite{gyongy-86}, we introduce the function
$\tb(t,x)$ defined by~\eqref{eq:tb}, which reads in the present context
as
\begin{equation}
\label{eq:tb_2d}
\tb(t,x) = \int_\rens \left[ - \partial_x V(x,y)\right] 
\psi^x_{\rm cond}(t,y) \, dy
=
- \frac{\dis
\int_\rens \partial_x V(x,y) \, \psi(t,x,y) \,
dy
}{\dis \int_\rens \psi(t,x,y) \, dy}
\end{equation}
and the resulting dynamics~\eqref{eq:ty} reads
\begin{equation}
\label{eq:ty_2d}
d \widetilde{X}_t = \tb(t, \widetilde{X}_t) \, dt + \sqrt{2 \beta^{-1}}
\ dW^x_t.
\end{equation}

We now prove Lemma~\ref{lem:gyongy} in that specific context and show
that, at any time $t$, the probability distribution function of $\widetilde{X}_t$ is
equal to that of $\xi(Q_t) = X_t$. 

\begin{proof}[Lemma~\ref{lem:gyongy}, case $\xi(x,y)=x$]
\smartqed
The probability density function $\psi(t,x,y)$ of $Q_t = (X_t,Y_t)$
satisfies the Fokker-Planck equation~\eqref{eq:FP_complet}:
\begin{eqnarray}
\nonumber
\partial_t \psi &=& 
\dive \left( \psi \nabla V \right) + \beta^{-1} \Delta \psi
\\
\label{eq:FP_2d}
&=&
\partial_x \left( \psi \partial_x V \right) + 
\partial_y \left( \psi \partial_y V \right) + 
\beta^{-1} \partial_{xx} \psi + \beta^{-1} \partial_{yy} \psi.
\end{eqnarray}
The probability distribution function of $\xi(Q_t) = X_t$ is 
$$
\psi^\xi(t,x) = \int_\rens \psi(t,x,y) \, dy.
$$ 
Integrating~\eqref{eq:FP_2d} with respect to $y$, we obtain 
\begin{eqnarray}
\nonumber
\partial_t \psi^\xi &=& 
\partial_x \left( \int \psi \partial_x V \, dy \right) + 
\beta^{-1} \partial_{xx} \psi^\xi 
\\
\label{eq:titi}
&=&
- \partial_x \left( \psi^\xi \ \tb \right) + 
\beta^{-1} \partial_{xx} \psi^\xi ,
\end{eqnarray}
where $\tb(t,x)$ is given by~\eqref{eq:tb_2d}. We recognize the
Fokker-Planck equation associated to the equation~\eqref{eq:ty_2d}.  
\qed
\end{proof}

As pointed out above, \eqref{eq:ty} ({\em i.e.}~\eqref{eq:ty_2d} here) cannot be considered as a reasonable
closure, since it involves the function $\tb$, which is defined using
$\psi(t,x,y)$ (see~\eqref{eq:tb_2d}), which in practice is hardly computable. We thus
approximate $\tb$ by the function $b$ defined by~\eqref{eq:b}, which
 amounts to replacing $\psi(t,x,y)$ in~\eqref{eq:tb_2d} by the
equilibrium measure $\psi_\infty(x,y)$:
$$
b(x) = 
- \frac{\dis
\int_\rens \partial_x V(x,y) \ \psi_\infty(x,y) \,
dy
}{\dis \int_\rens \psi_\infty(x,y) \, dy}.
$$
In the spirit of~\eqref{eq:ty_2d}, we thus introduce the effective dynamics 
\begin{equation}
\label{eq:y_2d}
d \overline{X}_t = b(\overline{X}_t) \, dt + \sqrt{2 \beta^{-1}}
\ dW^x_t.
\end{equation}

\medskip

We now prove Proposition~\ref{prop:D3} (error estimator on the effective
dynamics), in the specific case at hand
here.
The assumption~\eqref{H2} means that the measure~\eqref{eq:mu_z_2d}
satisfies, for any $z$, a logarithmic Sobolev inequality with a constant
$\rho$ independent of $z$. The assumption~\eqref{H3} reads
$\kappa = \| \partial_{xy} V \|_{L^\infty} < \infty$,
and the assumption~\eqref{H4} is satisfied with $\lambda = 0$ since
$\nabla \xi = (1,0)^T$ is a constant vector.

\begin{proof}[Proposition~\ref{prop:D3}, case $\xi(x,y) = x$]
\smartqed
By definition (see~\eqref{eq:A}), the free energy $A$ associated to the
reaction coordinate $\xi$ satisfies 
$$
\exp(-\beta A(x)) = \int_\rens \psi_\infty(x,y) \, dy = Z^{-1} \int_\rens \exp(-\beta
V(x,y)) \, dy
$$
hence
\begin{equation}
\label{eq:A_prime_2d}
A'(x) = 
\frac{ \dis
\int \partial_x V(x,y) \psi_\infty(x,y) \,
dy
}{\dis \int_\rens \psi_\infty(x,y) \, dy}
= -b(x).
\end{equation}
The effective dynamics~\eqref{eq:y_2d} thus reads
$$
d\overline{X}_t = -A'(\overline{X}_t) \, dt + \sqrt{2/\beta} \, dW^x_t.
$$
Note that, in this specific context, the 
effective dynamics is of the form~\eqref{eq:zbar} (see~\cite[Section
2.3]{dyn_eff} for a comprehensive discussion of the relation between the
effective dynamics and~\eqref{eq:zbar}).  
The probability distribution $\phi(t,x)$ of $\overline{X}_t$ satisfies
the Fokker-Planck equation associated to the above stochastic
differential equation, that reads
\begin{equation}
\label{eq:titi2}
\partial_t \phi =
\partial_x \left( \phi \ A' \right) + 
\beta^{-1} \partial_{xx} \phi .
\end{equation}
Consider now the relative entropy 
$$
E(t) = H(\psi^\xi | \phi) = 
\int_\rens \ln \left(\frac{\psi^\xi(t,x)}{\phi(t,x)} \right) \psi^\xi(t,x)
\, dx.
$$
We compute, using~\eqref{eq:titi2} and~\eqref{eq:titi}, that
\begin{eqnarray*}
\frac{dE}{dt}
&=& 
\int_\rens \ln \left(\frac{\psi^\xi}{\phi} \right) 
\partial_t \psi^\xi 
-
\int_\rens \frac{\psi^\xi}{\phi} \ \partial_t \phi 
\\
&=& 
\int_\rens \ln \left(\frac{\psi^\xi}{\phi} \right) 
\left[ 
- \partial_x \left( \psi^\xi \ \tb \right) + 
\beta^{-1} \partial_{xx} \psi^\xi
\right]
-
\int_\rens \frac{\psi^\xi}{\phi} \ \left[
\partial_x \left( \phi \ A' \right) + 
\beta^{-1} \partial_{xx} \phi 
\right]
\\
&=& 
-\beta^{-1} \int_\rens \partial_x \left[ \ln \left(\frac{\psi^\xi}{\phi} \right)
\right] \partial_x \psi^\xi
+ \beta^{-1}
\int_\rens \partial_x \left( \frac{\psi^\xi}{\phi} \right) \partial_x
\phi 
\\
&& \hspace{2cm}
+ \int_\rens \psi^\xi \
\partial_x \left( \ln \frac{\psi^\xi}{\phi} \right) \left( \tb + A' \right)
\\
&=& 
-\beta^{-1} \int_\rens \partial_x \left[ \ln \left(\frac{\psi^\xi}{\phi} \right)
\right] 
\left[ \partial_x \psi^\xi - \frac{\psi^\xi \partial_x \phi}{\phi} \right]
+ \int_\rens \psi^\xi \
\partial_x \left( \ln \frac{\psi^\xi}{\phi} \right) \left( \tb + A' \right)
\\
&=& 
-\beta^{-1} \int_\rens \partial_x \left[ \ln \left(\frac{\psi^\xi}{\phi} \right)
\right] \phi \ \partial_x \left( \frac{\psi^\xi}{\phi} \right) 
+ \int_\rens \psi^\xi \
\partial_x \left( \ln \frac{\psi^\xi}{\phi} \right) \left( \tb + A' \right)
\\
&=& -\beta^{-1} I(\psi^\xi | \phi) + \int_\rens \psi^\xi \
\partial_x \left( \ln \frac{\psi^\xi}{\phi} \right) \left( \tb + A' \right).
\end{eqnarray*}
Using a Young inequality with a parameter $\alpha>0$ to be fixed later,
we obtain 
\begin{eqnarray}
\nonumber
\frac{dE}{dt}
&\leq &
- \beta^{-1} I(\psi^\xi | \phi) + \frac{1}{2 \alpha} \int_\rens \psi^\xi
\left( \partial_x \left( \ln \frac{\psi^\xi}{\phi} \right) \right)^2
+ \frac{\alpha}{2} \int_\rens \psi^\xi \left( A' + \tb \right)^2
\\
&=&
\label{eq:yes}
\left(  \frac{1}{2 \alpha} - \beta^{-1} \right) I(\psi^\xi | \phi) +
\frac{\alpha}{2} \int_\rens \psi^\xi \left( A' + \tb \right)^2.
\end{eqnarray}
We now observe that, in view of~\eqref{eq:tb_2d}
and~\eqref{eq:A_prime_2d}, $A'$ and $-\tb$ are averages of the {\em 
  same} quantity with respect to different probability measures:
$$
-\tb(t,x) = \int_\rens \partial_x V(x,y) \, \nu_1^{t,x}(y) \, dy 
\quad \text{and} \quad 
A'(x) = \int_\rens \partial_x V(x,y) \, \nu_2^x(y) \, dy
$$
with 
\begin{equation}
\label{eq:def_nu}
\nu_1^{t,x}(y) = \frac{\psi(t,x,y)}{\int_\rens \psi(t,x,y) \, dy}
\quad \text{and} \quad 
\nu_2^x(y) = \frac{\psi_\infty(x,y)}{\int_\rens \psi_\infty(x,y) \, dy}.
\end{equation}
We write
\begin{eqnarray*}
A'(x) + \tb(t,x) &=& \int_\rens \partial_x V(x,y) \, \nu_2^x(y) \, dy - 
\int_\rens \partial_x V(x,y) \, \nu_1^{t,x}(y) \, dy
\\
&=&
\int_{\rens^2} \left( \partial_x V(x,y_1) - \partial_x V(x,y_2) \right)
k^{t,x}(y_1,y_2) \ dy_1 \, dy_2
\end{eqnarray*}
for any probability measure $k^{t,x}$ such that 
$$
\int_\rens k^{t,x}(y_1,y_2) \ dy_2 = \nu_2^x(y_1) 
\quad \text{and} \quad
\int_\rens k^{t,x}(y_1,y_2) \ dy_1 = \nu_1^{t,x}(y_2).
$$
Hence,
\begin{eqnarray*}
\left| A'(x) + \tb(t,x) \right|
& \leq & 
\| \partial_{xy} V \|_{L^\infty}
\int_{\rens^2} \left| y_1 - y_2 \right|
k^{t,x}(y_1,y_2) \ dy_1 \, dy_2
\\
& \leq & 
\| \partial_{xy} V \|_{L^\infty}
\left( \int_{\rens^2} \left| y_1 - y_2 \right|^2
k^{t,x}(y_1,y_2) \ dy_1 \, dy_2 \right)^{1/2}.
\end{eqnarray*}
We now optimize on $k^{t,x}$. Introducing the Wasserstein distance
$W(\nu_1^{t,x},\nu_2^x)$ between $\nu_1^{t,x}$ and $\nu_2^x$
(see~\eqref{eq:wasser}), we obtain 
$$
\left| A'(x) + \tb(t,x) \right|
\leq
\| \partial_{xy} V \|_{L^\infty} \ W(\nu_1^{t,x},\nu_2^x).
$$
As recalled above, assumption~\eqref{H2} means that $\nu_2^x$ satisfies
a Logarithmic Sobolev inequality. Thus, it also 
satisfies a Talagrand inequality (see Lemma~\ref{lem:lsi-t}), hence
$$
W(\nu_1^{t,x},\nu_2^x) \leq 
\sqrt{ \frac{2}{\rho} H(\nu_1^{t,x} | \nu_2^x)}
\leq
\frac{1}{\rho} \sqrt{I(\nu_1^{t,x} | \nu_2^x)}.
$$
As a consequence, 
$$
\left| A'(x) + \tb(t,x) \right|
\leq
\frac{\| \partial_{xy} V \|_{L^\infty}}{\rho} \ \sqrt{ I(\nu_1^{t,x}|\nu_2^x) }.
$$
Using~\eqref{eq:def_nu}, we obtain
\begin{eqnarray*}
\int_\rens \psi^\xi \left( A' + \tb \right)^2 \, dx
&\leq& 
\frac{\| \partial_{xy} V \|^2_{L^\infty}}{\rho^2} 
\int_\rens \psi^\xi(t,x) \ I(\nu_1^{t,x}|\nu_2^x) \, dx
\\
&\leq&
\frac{\| \partial_{xy} V \|^2_{L^\infty}}{\rho^2} 
\int_\rens \psi^\xi(t,x) \left[ \int_\rens 
\left| \partial_y \ln \frac{\psi(t,x,y)}{\psi_\infty(x,y)} \right|^2 
\frac{\psi(t,x,y)}{\psi^\xi(t,x)} \ dy \right] dx
\\
&\leq&
\frac{\| \partial_{xy} V \|^2_{L^\infty}}{\rho^2} \
I(\psi |\psi_\infty).
\end{eqnarray*}
Returning to~\eqref{eq:yes}, and using~\eqref{eq:ici}, we thus deduce that
\begin{eqnarray*}
\frac{dE}{dt} &\leq& 
\left( \frac{1}{2 \alpha} - \beta^{-1} \right) I(\psi^\xi | \phi) +
\frac{\alpha}{2} \ \frac{\| \partial_{xy} V \|^2_{L^\infty}}{\rho^2}
I(\psi |\psi_\infty) 
\\
&=&
\left(  \frac{1}{2 \alpha} - \beta^{-1} \right) I(\psi^\xi | \phi) -
\frac{\alpha \beta \| \partial_{xy} V \|^2_{L^\infty}}{2\rho^2}
\partial_t H(\psi |\psi_\infty).
\end{eqnarray*}
We take $2 \alpha = \beta$, so that the first term vanishes, and we are
left with
$$
\frac{dE}{dt} \leq 
-\frac{\beta^2 \| \partial_{xy} V \|^2_{L^\infty}}{4\rho^2}
\partial_t H(\psi |\psi_\infty).
$$
Integrating this inequality between the times 0 and $t$, and using that
$E(0) = 0$, we obtain 
$$
E(t) \leq 
\frac{\beta^2 \| \partial_{xy} V \|^2_{L^\infty}}{4\rho^2}
\left( H(\psi(t=0) |\psi_\infty) - H(\psi(t,\cdot) |\psi_\infty) \right).
$$
As recalled above, assumption~\eqref{H3} reads 
$\kappa = \| \partial_{xy} V \|_{L^\infty} < \infty$. 
The above bound is thus exactly the bound~\eqref{eq:D3estim} in the present
context. 
\qed
\end{proof}

\subsection{Numerical results}
\label{sec:eff_dyn_num}

In this section, we check the accuracy of the effective
dynamics~\eqref{eq:y} in terms of residence times, and also compare this
effective dynamics with the coarse-grained dynamics~\eqref{eq:zbar} based 
on the free energy. We perform such comparison on two test-cases, and
evaluate the influence of the temperature on the results. We also
provide some analytical explanations for the observed numerical
results. 

In the following numerical tests, we focus on residence times. 
We have indeed already underlined that the characteristic behaviour of the
dynamics~\eqref{eq:X} is to sample a given well of the potential energy,
then suddenly hope to another basin, and start over. Consequently, an
important quantity is the residence time that the system spends in the
well, before going to another one. 

\medskip

For all the numerical tests reported in this section, 
the complete dynamics~\eqref{eq:X} has been integrated with
the Euler-Maruyama scheme
$$
X_{j+1} = X_j -  \Delta t \, \nabla V(X_j) + \sqrt{2 \, \Delta t \,
  \beta^{-1}} \ G_j,
$$
where, for any $j$, $G_j$ is a $n$-dimensional vector, whose coordinates are
independent and identically distributed (i.i.d.) random variables,
distributed according to a normal Gaussian law. 

For the simulation of the dynamics~\eqref{eq:y} and~\eqref{eq:zbar}, we
need to have an expression for the free energy derivative $A'$ and the
functions $b$ and $\sigma$. These have been computed using the algorithm
proposed in~\cite{c_l_eve}, on a regular grid of some bounded
interval. Values of the functions for points that do not belong to that 
grid were obtained by linear interpolation. We have again used the
Euler-Maruyama 
scheme to numerically integrate the dynamics~\eqref{eq:y} and~\eqref{eq:zbar}. 

To compute residence times in a well, we have proceeded as follows (for
the sake of clarity, we assume in the following that there are only two
wells in the test case at hand). First, the left and the right
wells are defined as the sets $\left\{q \in \R^n; \ \xi(q) \leq
  \xi^{\rm th}_{\rm left} \right\}$ and $\left\{q \in \R^n; \ \xi(q) \geq 
  \xi^{\rm th}_{\rm right} \right\}$ respectively, with 
$\xi^{\rm th}_{\rm right} > \xi^{\rm th}_{\rm left}$. Next, we perform
the following computations: 
\begin{enumerate}
\item we first generated a large number ${\mathcal N}$ of configurations 
$\{ q_i \in \rens^n\}_{1 \leq i \leq {\mathcal N}}$, distributed
according to the measure $\mu$ restricted to the right well: as a
consequence, $\xi(q_i) > \xi^{\rm th}_{\rm right}$.
\item we next ran the dynamics~\eqref{eq:X} from the initial condition
  $q_i$, and monitor the first time $\tau_i$ at which the
  system reaches a point $q(\tau_i)$ in the left well:
$\tau_i = \inf \left\{ t; \ \xi(q_t) < \xi^{\rm th}_{\rm left} \right\}$.
\item from these $(\tau_i)_{1 \leq i \leq {\mathcal N}}$, we computed an average
  residence time and a confidence interval. These figures are
  the reference figures. 
\item we next consider the initial conditions 
$\left\{ \xi(q_i) \in \R \right\}_{1 \leq i \leq {\mathcal N}}$
for the effective dynamics. By construction, these
configurations are distributed according to the equilibrium measure
$\xi \star \mu$ (that is $\exp(-\beta A(z)) \, dz$) restricted to the right well. 
\item from these initial conditions, we run the dynamics~\eqref{eq:y}
  or~\eqref{eq:zbar} until the left well is reached, and compute, as 
  for the complete description, a residence time and its confidence
  interval.
\end {enumerate}

\subsubsection{A three atom molecule}
\label{sec:3at}

Our aim in this section is to show that different reaction coordinates, although similar
at first sight, can lead to very different results. As explained
in~\cite{dyn_eff}, the error estimate~\eqref{eq:D3estim} can then help discriminating
between these reaction coordinates.

We consider here a molecule made of three two-dimensional particles, whose
positions are $q_A$, $q_B$ and $q_C$. The potential energy of the system
is 
\begin{equation}
\label{eq:V3}
V(q) = \frac{1}{2\eps} \, \left(r_{AB} - \ell_{\rm eq} \right)^2 + 
\frac{1}{2\eps} \, \left(r_{BC} - \ell_{\rm eq} \right)^2
+ W_3(\theta_{ABC}),
\end{equation}
where $r_{AB} = \| q_A - q_B \|$ is the distance between atoms A and B,
$\ell_{\rm eq}$ is an equilibrium distance, $\theta_{ABC}$ is the angle
formed by the three atoms, and $W_3(\theta)$ is a three-body potential,
that we choose here to be a double-well potential: 
$$
W_3(\theta) = \frac12 k_\theta \, ((\theta-\theta_{\rm saddle})^2-\delta \theta^2)^2.
$$
Wells of $W_3$ are located at $\theta = \theta_{\rm saddle} \pm \delta
\theta$. The potential~\eqref{eq:V3} represents stiff bonds between particles A and~B
on the one hand, and B and C on the other hand, with a softer term
depending on the angle $\theta_{ABC}$. To remove rigid body motion
invariance, we set $q_B = 0$ and $q_A \cdot e_y = 0$.  
In the following, we work with the parameters $\eps = 10^{-3}$, $k_\theta = 208$,
$\ell_{\rm eq} = 1$, $\theta_{\rm saddle} = \pi/2$ and $\delta \theta =
\theta_{\rm saddle} - 1.187$. All dynamics are integrated with the time 
step $\Delta t = 10^{-3}$. 

We consider two reaction coordinates, that both indicate in which well
the system is: 
\begin{itemize}
\item the angle formed by the three atoms: 
$$
\xi_1 = \theta_{ABC}.
$$ 
In that case, wells are defined by
$\left\{q \in \R^n; \ \xi_1(q) \leq \xi^{\rm th}_{\rm left} \right\}$ and
$\left\{q \in \R^n; \ \xi_1(q) \geq \xi^{\rm th}_{\rm right} \right\}$,
with $\xi^{\rm th}_{\rm left} = \theta_{\rm saddle} - 0.15$ and 
$\xi^{\rm th}_{\rm right} = \theta_{\rm saddle} + 0.15$.  
\item the square of the distance between $A$ and $C$: 
$$
\xi_2 = \| q_A - q_C \|^2. 
$$
In that case, wells are defined by
$\left\{q \in \R^n; \ \xi_2(q) \leq \xi^{\rm th}_{\rm left} \right\}$ and
$\left\{q \in \R^n; \ \xi_2(q) \geq \xi^{\rm th}_{\rm right} \right\}$,
with $\xi^{\rm th}_{\rm left} = 1.6 \ell_{\rm eq}^2$ and 
$\xi^{\rm th}_{\rm right} = 2.4 \ell_{\rm eq}^2$. 
\end{itemize}
Note that there is a region of state space that does not belong to any
well. This choice allows to circumvent the so-called recrossing
problem. 

\begin{remark}
Note that~\eqref{eq:V3} reads
$$
V(q) = 
\frac{1}{2\eps} \left( U_{AB}(q)^2 + U_{BC}(q)^2 \right)
+ W_3(\theta_{ABC})
$$
with $U_{AB}(q) = r_{AB} - \ell_{\rm eq}$ and $U_{BC}(q) = r_{BC} -
\ell_{\rm eq}$. The two first terms in $V$ are much stiffer than the last one. 
We observe that $\nabla \theta_{ABC} \cdot \nabla U_{AB} = 
\nabla \theta_{ABC} \cdot \nabla U_{BC} = 0$. Hence, the reaction
coordinate $\xi_1$ is orthogonal to the stiff terms of the potential 
energy, in contrast to $\xi_2$. In view of~\cite[Section 3.2]{dyn_eff},
we hence expect to obtain accurate results with $\xi_1$, in contrast to
$\xi_2$. This is indeed the case, as shown in the sequel of this
section. 
$\diamond$
\end{remark}

We compute
the residence time in a given well following the complete description,
and compare it with the result given by a reduced description, based
either on~\eqref{eq:y} or~\eqref{eq:zbar}. Results are gathered in
Table~\ref{tab:residence_3at}, for the temperatures $\beta^{-1} = 1$ and
$\beta^{-1} = 0.2$. We observe that working with $\xi_1$ (and
either~\eqref{eq:y} or~\eqref{eq:zbar}) leads to very accurate
results, independently of the temperature. On the other hand, when the
reaction coordinate is not 
orthogonal to the stiff terms of the potential, both coarse-grained dynamics
turn out to be not accurate. 

\begin{remark}
\label{rem:dyn_identiques}
In the case at hand here, $\| \nabla \xi_1 \|^2 = \| \nabla \theta_{ABC} \|^2
= r_{BC}^{-2}$. This quantity is almost a constant, since the bond
length potential is stiff and the temperature is small. Hence, along the
trajectory, we have that $\| \nabla \xi_1 \|^2 \approx \ell_{\rm
  eq}^{-2} = 1$. This explains why, when choosing the reaction
coordinate $\xi_1$, dynamics~\eqref{eq:y} and~\eqref{eq:zbar} give
similar results.
$\diamond$
\end{remark}

\begin{table}[htbp]
\centerline{
\begin{tabular}{| c | c | c | c | c |}
  \hline
Temperature & Reaction & Reference & Residence time & Residence time
\\
& coordinate & residence time & using~\eqref{eq:y} & using~\eqref{eq:zbar}
\\
\hline \hline
$\beta^{-1} = 1$ 
& $\xi_1 = \theta_{ABC}$ & 0.700 $\pm$ 0.011 & 0.704 $\pm$ 0.011
& 0.710 $\pm$ 0.011
\\
\hline
$\beta^{-1} = 1$
& $\xi_2 = r_{AC}^2 $ & 0.709 $\pm$ 0.015 & 0.219 $\pm$ 0.004
& 2.744 $\pm$ 0.056 
\\
\hline \hline
$\beta^{-1} = 0.2$ 
& $\xi_1 = \theta_{ABC}$ & 5784 $\pm$ 101 & 5836 $\pm$ 100 & 5752 $\pm$ 101
\\
\hline
$\beta^{-1} = 0.2$
& $\xi_2 = r_{AC}^2 $ & 5833 $\pm$ 88 & 1373 $\pm$ 20 & 2135 $\pm$ 319 
\\
\hline
\end{tabular}
}
\caption{Three-atom molecule: residence times obtained from the complete
  description (third 
  column) and from the reduced descriptions (two last columns), for both
  reaction coordinates (confidence intervals have been computed on the
  basis of ${\cal N} = 15000$ realizations).} 
\label{tab:residence_3at}
\end{table}
% beta=1: je reprend la table de 
% /home/legoll/legoll/participations_congres/09sept_ln3m_lyon/talk

% beta = 5: je reprend les fichiers de 
% /home/legollf/legoll/stagiaires/Jan_Stage_scientifique_ENPC_mai06/h2o_code1/basse_temp/distance
% et
% /home/legollf/legoll/stagiaires/Jan_Stage_scientifique_ENPC_mai06/h2o_code1/basse_temp/angle

We now study how results depend on temperature. Let us first consider
the reaction coordinate $\xi_1 = \theta_{ABC}$. Results are shown on
Fig.~\ref{fig:angle}. Both coarse-grained dynamics provide extremely
accurate results, independently of the temperature. We also observe that
we can fit the residence time $\tau_{\rm res}$ according to the relation
\begin{equation}
\label{eq:tau_res_num}
\tau_{\rm res} \approx \tau^0_{\rm res} \, \exp(s \beta)
\end{equation}
with $\tau^0_{\rm res} = 0.07521$ and $s = 2.25031$. 

\begin{figure}[htbp]
\centerline{
\input{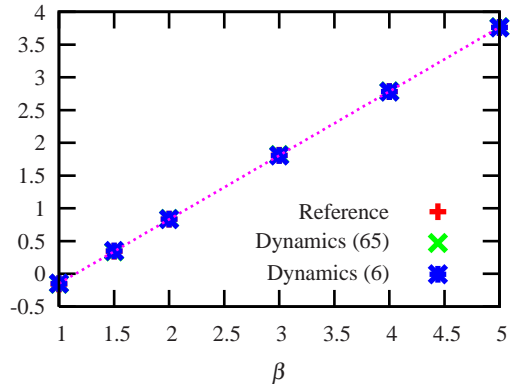}
}
\caption{$\log_{10} ({\rm residence \ time})$ as a function of $\beta$, for
  the reaction coordinate $\xi_1 = \theta_{ABC}$.} 
\label{fig:angle}
\end{figure}

\medskip

By analytical considerations, we now explain why the residence times
computed from both coarse-grained dynamics~\eqref{eq:zbar}
and~\eqref{eq:y} satisfy the 
relation~\eqref{eq:tau_res_num}, with the numerical values of $s$ and
$\tau^0_{\rm res}$ reported above.

We first consider the coarse-grained dynamics~\eqref{eq:zbar}
  driven by the free energy.
In the case at hand here, it is possible to compute analytically the
free energy. Using the internal coordinates $r_{AB}$, $r_{BC}$ and
$\theta_{ABC}$, we indeed infer 
from~\eqref{eq:stat} that the free energy $A_1$ does not depend on the
temperature and satisfies
$$
A_1(\theta_{ABC}) = W_3(\theta_{ABC}). 
$$
Thus $A_1$ has two global minimizers, separated by a barrier 
$$
\Delta A_1 = \frac12 k_\theta (\delta \theta)^4 \approx 2.25648.
$$

The large deviation theory can be used to understand the behaviour of
the dynamics~\eqref{eq:zbar}, in the low temperature regime. It yields
the fact that, when $\beta \gg 1$, residence times are given by
\begin{equation}
\label{eq:tau_res_ld}
\tau^{\rm LD}_{\rm res} \approx 
\tau^{\rm 0, LD}_{\rm res} \, \exp(\beta \Delta A_1)
\quad \text{with} \quad
\tau^{\rm 0, LD}_{\rm res} = 
\frac{2 \pi}{\omega_{\rm SP} \, \omega_{\rm W}},
\end{equation}
where $\omega_{\rm SP} = \sqrt{-A_1''(\xi_{\rm SP})}$ is the pulsation
at the saddle-point $\xi_{\rm SP} = \theta_{\rm saddle}$, and $\omega_{\rm W} =
\sqrt{A_1''(\xi_{\rm W})}$ is the pulsation at the local minimizer
$\xi_{\rm W} = \theta_{\rm saddle} \pm \delta \theta$ (see
also~\cite[Eqs. (7.9) and (7.10)]{kramers_review}). In the present case,
we compute that $\omega_{\rm SP} \approx 7.828$ and $\omega_{\rm W}
\approx 11.07$, thus $\tau^{\rm 0, LD}_{\rm res} \approx 0.0725$, and we
find that
$$
s \approx \Delta A_1 
\quad \text{and} \quad
\tau^0_{\rm res} \approx \tau^{\rm 0, LD}_{\rm res}.
$$
We thus obtain a good agreement between~\eqref{eq:tau_res_num}
and~\eqref{eq:tau_res_ld}, as observed on Fig.~\ref{fig:angle}. Note
that this agreement holds even up 
to temperature $\beta^{-1} = 1$. 

\smallskip

We now turn to the dynamics~\eqref{eq:y}. We pointed out in
Remark~\ref{rem:dyn_identiques} that dynamics~\eqref{eq:y}
and~\eqref{eq:zbar} are identical in the limit of low temperature. The
functions $b$ and $\sigma$ are plotted for the temperature $\beta^{-1} =
1$ on Fig.~\ref{fig:b_sigma}. We observe that, even though the
temperature is not very small, we already have $b \approx -W'_3 = -
A'_1$ and $\sigma \approx 1$. The agreement is even better when the
temperature is smaller. This thus explains why results given by both
coarse-grained dynamics~\eqref{eq:y} and~\eqref{eq:zbar} can be fitted
by the same relation~\eqref{eq:tau_res_num}, on the whole range of
temperature. 

\begin{figure}[htbp]
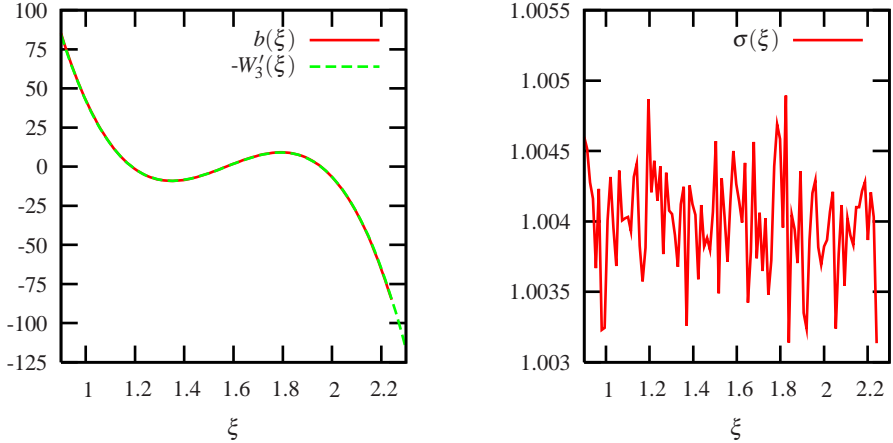

\centerline{
\input{angle_drift.tex} \input{angle_sigma.tex}
}
\caption{Plot of the functions $b$ and $\sigma$, for the
  reaction coordinate $\xi_1 = \theta_{ABC}$, at the temperature
  $\beta^{-1}=1$.}
\label{fig:b_sigma}
\end{figure}

\bigskip

We now consider the reaction coordinate $\xi_2 = r^2_{AC}$. Residence
times as a function of the inverse temperature $\beta$ are shown on
Fig.~\ref{fig:distance}. We observe that neither the
dynamics~\eqref{eq:zbar} nor the dynamics~\eqref{eq:y} provide accurate
results. More precisely, the reference results, the results given
by~\eqref{eq:y} and the results given by~\eqref{eq:zbar} can be fitted
by 
\begin{eqnarray}
\nonumber
\tau_{\rm res}^{\rm ref} &\approx& \tau^{\rm 0,ref}_{\rm res} \, \exp(s \beta),
\\
\label{eq:tau_res_num_distance_eff}
\tau_{\rm res}^{\rm eff} &\approx& \tau^{\rm 0,eff}_{\rm res} \, \exp(s \beta),
\\
\label{eq:tau_res_num_distance_free}
\tau_{\rm res}^{\rm free} &\approx& \tau^{\rm 0,free}_{\rm res} \, \exp(s \beta)
\end{eqnarray}
respectively, with the same parameter $s = 2.21 \pm 0.03$ and 
$$
\tau^{\rm 0,ref}_{\rm res} \approx 0.0768,
\quad
\tau^{\rm 0,eff} \approx 0.0241,
\quad
\tau^{\rm 0,free} \approx 0.293.
$$
The dependency with respect to the temperature is thus accurately
reproduced by both coarse-grained dynamics. The inaccuracy comes from
the fact that the prefactor $\tau^{\rm 0,ref}$ is ill-approximated. 

\begin{figure}[htbp]
\centerline{
\input{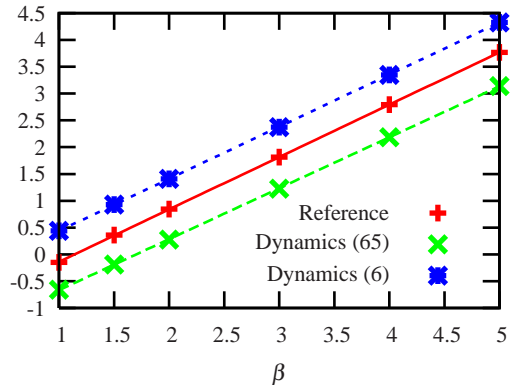}
}
\caption{$\log_{10} ({\rm residence \ time})$ as a function of $\beta$,
  for the reaction coordinate $\xi_2 = r^2_{AC}$.} 
\label{fig:distance}
\end{figure}

Again, these numerical observations are in agreement with analytical
computations based on the large deviation theory. 
More precisely, we explain in the sequel why the residence times
  computed from both coarse-grained dynamics~\eqref{eq:y} and~\eqref{eq:zbar}
  satisfy~\eqref{eq:tau_res_num_distance_eff}
  and~\eqref{eq:tau_res_num_distance_free}, with the same $s$, and for
  the numerical values of $s$, $\tau^{\rm 0,eff}$ and $\tau^{\rm
    0,free}$ reported above.

The functions $A_2$,
$b$ and $\sigma$ are plotted for two different temperatures on
Fig.~\ref{fig:b_sigma_distance}. Although $A_2$ a priori depends on
$\beta$ (as expected), it turns out this dependency becomes quite weak
when $\beta \geq 1$. It turns out that we can fit $A'_2$ by
$$
A'_2(\xi) \approx c_5 (x-2)^5 + c_4 (x-2)^4 + c_3 (x-2)^3 + c_2 (x-2)^2 + c_1
(x-2),
$$
with $
c_1 = -16.4433, \
c_2 = 3.87398, \
c_3 = 34.2171, \
c_4 = -6.36938$ and 
$c_5 = -7.89431$. The free energy has thus two local minimizers, 
$\xi_{\rm W,r} \approx 2.73$ and $\xi_{\rm W,l} \approx 1.25$ and a
saddle point, $\xi_{\rm SP} \approx 2$, with
$$
A_2(\xi_{\rm SP}) \approx 0, 
\quad 
A_2(\xi_{\rm W,r}) \approx -2.1,
\quad
A_2(\xi_{\rm W,l}) \approx -2.37.
$$

\begin{figure}[htbp]
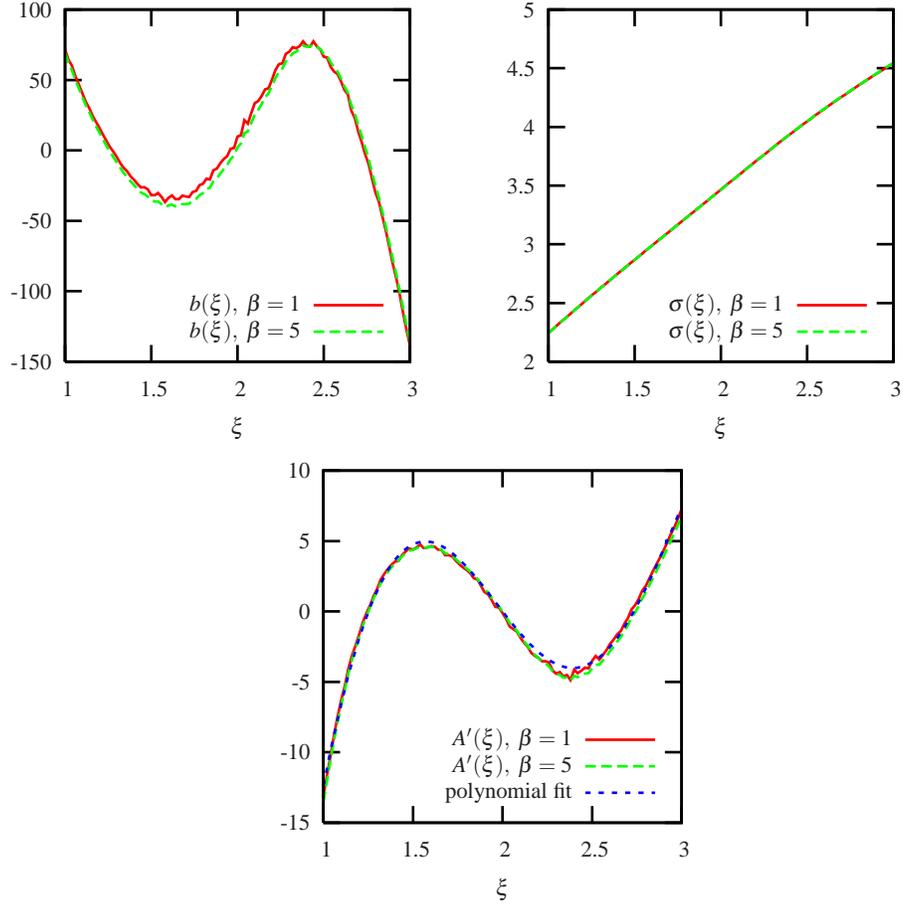

\centerline{
\input{distance_drift.tex} \input{distance_sigma.tex} 
}
\hspace{3cm}
\input{distance_aprime.tex}
\caption{Plot of the functions $b$, $\sigma$ and $A'_2$, for the
  reaction coordinate $\xi_2 = r^2_{AC}$, at two different
  temperatures.}
\label{fig:b_sigma_distance}
\end{figure}

We introduce the barriers to go from the right well to the left well (r
$\to$ l) and {\em vice-versa}:
$$
\Delta A_2^{\rm r \to l} = A_2(\xi_{\rm SP}) - A_2(\xi_{\rm W,r})
\quad \text{and} \quad
\Delta A_2^{\rm l \to r} = A_2(\xi_{\rm SP}) - A_2(\xi_{\rm W,l}).
$$
In the case of the dynamics~\eqref{eq:zbar} driven by the free energy,
and under the assumption that the temperature is low enough so that
$A_2$ becomes independent of $\beta$,
the large deviation theory can again be used, and yields the fact that
residence times are given by
$$
\tau^{\rm LD, r \to l}_{\rm res, free} \approx \frac{2 \pi}{\omega_{\rm SP} \,
  \omega_{\rm W,r}} \, \exp(\beta \Delta A_2^{\rm r \to l}),
\quad
\tau^{\rm LD, l \to r}_{\rm res, free} \approx \frac{2 \pi}{\omega_{\rm SP} \,
  \omega_{\rm W,l}} \, \exp(\beta \Delta A_2^{\rm l \to r}),
$$
where $\omega_{\rm SP}$, $\omega_{\rm W,l}$ and $\omega_{\rm W,r}$ are
the pulsations at the saddle-point, the left well and the right well,
respectively. In the present case,
we compute that $\omega_{\rm SP} \approx \sqrt{-c_1} \approx 4.055$, 
$\omega_{\rm W,l} \approx 5.809$ 
and $\omega_{\rm W,r} \approx 4.774$.

The left well is deeper than the right well. Hence, in the low
temperature limit, the residence time in the left well is much
larger than the residence time in the right well, and the probability to
be in the left well is higher than the probability to be in the right
well. Hence,
\begin{equation}
\label{eq:tau_res_ld_distance}
\tau^{\rm LD}_{\rm res, free} \approx 
\tau^{\rm LD, l \to r}_{\rm res, free} \approx
\tau^{\rm 0, LD, l \to r}_{\rm res, free} \, \exp(\beta \Delta A_2^{\rm l \to r})
\quad \text{with} \quad
\tau^{\rm 0, LD, l \to r}_{\rm res, free} = 
\frac{2 \pi}{\omega_{\rm SP} \, \omega_{\rm W,l}}.
\end{equation}
With the parameters that we used, we compute $\tau^{\rm 0, LD, l \to
  r}_{\rm res, free} \approx 0.267$, hence 
$$
s \approx \Delta A_2^{\rm l \to r}
\quad \text{and} \quad
\tau^{\rm 0,free}_{\rm res} \approx \tau^{\rm 0, LD, l \to r}_{\rm res, free},
$$
and we obtain a good agreement between~\eqref{eq:tau_res_num_distance_free}
and~\eqref{eq:tau_res_ld_distance}. 

\smallskip

We now turn to the dynamics~\eqref{eq:y}. The functions $b$ and $\sigma$
plotted on Fig.~\ref{fig:b_sigma_distance} seem to be almost independent
of the temperature when $\beta \geq 1$. Following~\cite[Section
2.3]{dyn_eff} and~\cite[Sec. 10 and Eq. (89)]{e-eve04}, we introduce the
one-to-one function $\dis h(\xi) = \int_0^\xi \sigma^{-1}(y) \, dy$ and
the coordinate $\zeta = h(\xi_2)$. We next change of variable in the
effective dynamics~\eqref{eq:y} on the reaction coordinate $\xi$ and
recast it as
$$
d\zeta_t = - \widetilde{A}'(\zeta_t) \, dt + \sqrt{2 \beta^{-1}} \, d B_t,
$$
where $\widetilde{A}$ turns out to be the free energy associated to the
reaction coordinate $\zeta(q) = h(\xi_2(q))$. The residence time to exit
the left well is hence given by 
$$
\tau^{\rm LD, l \to r}_{\rm res, eff} \approx
\frac{2 \pi}{\widetilde{\omega}_{\rm SP} \, \widetilde{\omega}_{\rm W,l}}
\, \exp(\beta \Delta \widetilde{A}^{\rm l \to r}).
$$
In the regime of low temperature, the second term of~\eqref{eq:F} is
negligible, and we deduce from~\eqref{eq:Aprime} that
$\widetilde{A}(h(\xi)) = A(\xi)$. As a consequence,
$$
\Delta \widetilde{A}^{\rm l \to r} = \Delta A^{\rm l \to r},
\quad
\widetilde{\omega}_{\rm SP} = \omega_{\rm SP} \, \sigma(\xi_{\rm SP}),
\quad
\widetilde{\omega}_{\rm W,l} = \omega_{\rm W,l} \, \sigma(\xi_{\rm W,l}).
$$
Hence,
\begin{equation}
\label{eq:tau_res_ld_distance_eff}
\tau^{\rm LD, l \to r}_{\rm res, eff} \approx
\tau^{\rm 0, LD, l \to r}_{\rm res, eff}
\, \exp(\beta \Delta A^{\rm l \to r})
\end{equation}
with
$$
\tau^{\rm 0, LD, l \to r}_{\rm res, eff}
=
\frac{2 \pi}{\omega_{\rm SP} \, \omega_{\rm W,l} \, \sigma(\xi_{\rm SP}) \, \sigma(\xi_{\rm W,l})}.
$$
We thus recover that the dependency of the residence times with
temperature is identical between the residence times predicted by the
effective dynamics~\eqref{eq:y} and the residence times predicted
by~\eqref{eq:zbar}: this dependency is exponential, with the same
prefactor $\Delta A^{\rm l \to r}$. 

We also compute $\sigma(\xi_{\rm SP}) \approx 3.465$ and 
$\sigma(\xi_{\rm W,l}) \approx 2.563$, so 
$\tau^{\rm 0, LD, l \to r}_{\rm res, eff} \approx 0.03$. Thus the values 
$\tau^{\rm 0,eff}_{\rm res}$ and $\tau^{\rm 0, LD, l \to r}_{\rm res,
  eff}$ qualitatively agree, and we obtain a good agreement
between~\eqref{eq:tau_res_num_distance_eff}
and~\eqref{eq:tau_res_ld_distance_eff}.  

\subsubsection{The butane molecule case}
\label{sec:butane}

We now consider a system in higher dimension, namely a butane molecule,
in the united atom model~\cite{rb78,ms98}. We hence only simulate four
particles, whose positions are $q^i \in \rens^3$, for $1 \leq i \leq
4$. The potential energy reads 
$$
V(q) = \sum_{i=1}^3 V_{\rm bond} \left( \| q^{i+1}-q^i \| \right) + 
V_{\rm bond-angle}(\theta_1) +
V_{\rm bond-angle}(\theta_2) +  V_{\rm torsion}(\phi),
$$
where $\theta_1$ is the angle formed by the three first particles, 
$\theta_2$ is the angle formed by the three last particles, and $\phi$
is the dihedral angle, namely the angle between the plane on which the
three first particles lay and the plane on which the three last
particles lay, with the convention $\phi \in (-\pi,\pi)$. We work with
$$
V_{\rm bond}(\ell) = \frac{k_2}2 (\ell-\ell_{eq})^2, 
\quad
V_{\rm bond-angle}(\theta) = \frac{k_3}2 (\theta-\theta_{eq})^2 
$$
and
$$
V_{\rm torsion}(\phi) = c_1 (1-\cos \phi) + 2 c_2 (1-\cos^2 \phi) + c_3
(1 + 3 \cos \phi -4 \cos^3 \phi).
$$
Rigid body motion invariance is removed by setting $q^2=0$, $q^1 \cdot
e_z = 0$ and $q^3 \cdot e_x = q^3 \cdot e_z = 0$. 

In the system of units where the length unit is $\ell_{0} = 1.53 
\cdot 10^{-10}$~m and the energy unit is such that $k_BT=1$ at
$T=300$~K, the time unit is $\bar{t} = 364$~fs, and  
the numerical values of the parameters are $\ell_{eq} = 1$, $k_3 =
208$, $\theta_{eq} = 1.187$, $c_1 = 1.18$, $c_2= -0.23$, and $c_3
=2.64$. We will work in the sequel with $k_2 = 1000$. We set the unit of
mass such that the mass of each particle is equal to~1.

For these values of the parameters $c_{i}$, the function $V_{\rm torsion}$ has a 
unique global minimum (at $\phi = 0$) and two local non-global
minima (see Fig.~\ref{fig:torsion}). It is hence a metastable potential.
We choose to work with the dihedral angle as reaction coordinate:
$$
\xi(q) = \phi.
$$ 
We are interested in the residence time in the main
well (around the global minimizer $\phi_0 = 0$) before hoping to any of
the two wells around the local minimizers $\phi_{\pm 1} = \pm 2 \pi
/3$. For each minimizer $\phi_0$, $\phi_1$ and $\phi_{-1}$, the
associated well is defined by 
$\left\{ q ; | \xi(q) - \phi_i | \leq \xi^{\rm th} \right\}$, 
$i=-1, 0, 1$, with $\xi^{\rm th} = 0.5$. 

\begin{remark}
We observe that
$$
\nabla V_{\rm stiff} \cdot \nabla \xi = 0,
$$
where $V_{\rm stiff}(q) = 
\sum_{i=1}^3 V_{\rm bond} \left( \| q^{i+1}-q^i \| \right) + V_{\rm
  bond-angle}(\theta_1) + V_{\rm bond-angle}(\theta_2)$. In view
of~\cite[Section 3.2]{dyn_eff}, we hence expect to obtain accurate
results with this choice of reaction coordinate, as it is indeed the
case.   
$\diamond$
\end{remark}

\begin{figure}[htbp]
\centerline{
\input{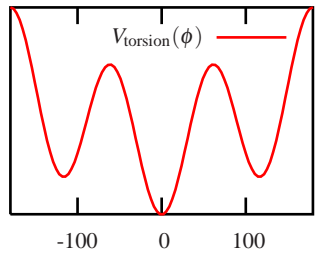}
}
\caption{Torsion angle potential $V_{\rm torsion}(\phi)$.} 
\label{fig:torsion}
\end{figure}

\medskip

As in the previous section, we compute reference residence times
by integrating the 
complete dynamics, and we then consider both coarse-grained
dynamics~\eqref{eq:y} and~\eqref{eq:zbar}. All computations have been
done with the time step $\Delta t = 10^{-3}$. Results are reported in
Table~\ref{tab:residence_butane}. We observe that the effective
dynamics~\eqref{eq:y} again yields extremely accurate results. The
results obtained by the dynamics~\eqref{eq:zbar}, although qualitatively
correct, are less accurate. This conclusion holds for all the
temperatures we considered.

\begin{table}[htbp]
\centerline{
\begin{tabular}{| c | c | c | c |}
  \hline
Temperature & Reference & Residence time & Residence time
\\
& residence time & using~\eqref{eq:y} & using~\eqref{eq:zbar}
\\
\hline 
$\beta^{-1} = 1$ & 31.9 $\pm$ 0.56 & 32.0 $\pm$ 0.56 & 37.1 $\pm$ 0.64
\\
\hline
$\beta^{-1} = 0.67$ & 493 $\pm$ 8 & 490 $\pm$ 8 & 581 $\pm$ 9
\\
\hline
$\beta^{-1} = 0.5$ & 7624 $\pm$ 113 & 7794 $\pm$ 115 & 9046 $\pm$ 133 
\\
\hline
\end{tabular}
}
\caption{Butane molecule: residence times obtained from the complete description (second
  column) and from the reduced descriptions (two last columns), at
  different temperatures (confidence intervals have been computed on the
  basis of ${\cal N} = 13000$ realizations).} 
% en fait, 13000 realisations ou plus
\label{tab:residence_butane}
\end{table}
% je reprend les resultats de 
% /home/legollf/legoll/stagiaires/Jan_Stage_scientifique_ENPC_mai06/butane_aout09/basse_temp

As in the previous section, residence times depend on the temperature
following 
$$
\tau_{\rm res} \approx \tau^0_{\rm res} \, \exp(s \beta).
$$
For both coarse-grained dynamics, the values found for $s$ and
$\tau^0_{\rm res}$ agree with predictions based on the large deviation
theory. In the case at hand here, it turns out that the free energy
associated to the reaction coordinate $\xi(q) = \phi$ is simply
$A(\xi) = V_{\rm torsion}(\xi)$. 
On Fig.~\ref{fig:b_sigma_butane}, we plot the functions $b$ and
$\sigma$. We observe that they are almost independent of the temperature (as
soon as $\beta \geq 1$), and that $\sigma$ is almost a
constant. Hence, up to the time rescaling $t_{\rm rescale} = \sigma t$,
the effective 
dynamics reads as the dynamics~\eqref{eq:zbar} governed by the free
energy. As $\sigma = 1.086 \approx 1$ (see
Fig.~\ref{fig:b_sigma_butane}), the dynamics~\eqref{eq:zbar} yields qualitatively correct
results. 

\begin{figure}[htbp]
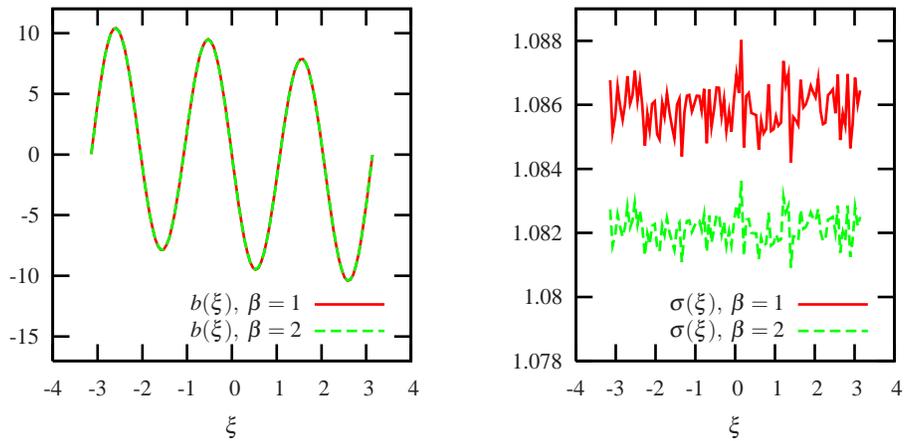

\centerline{
\input{dihedral_drift.tex} \input{dihedral_sigma.tex}
}
\caption{Plot of the functions $b$ and $\sigma$, for the
  reaction coordinate $\xi = \phi$, at different temperatures.}
\label{fig:b_sigma_butane}
\end{figure}

% je ne montre pas la courbe, on a deja bcp vu ce truc.
% \begin{figure}[htbp]
% \centerline{
% \input{dihedral_beta_temps.tex}
% }
% \caption{log of residence time as a function of $\beta$ (reaction
%   coordinate = angle dihedral).} 
% \label{fig:dihedral}
% \end{figure}

\begin{acknowledgement}
The present contribution is related to a lecture given by TL at 
a workshop at BIRS on ``Numerical analysis of multiscale
computations'' (December 7-11, 2009). 
This work is supported in part by the INRIA, under the grant ``Action
de Recherche Collaborative'' HYBRID, and by the Agence Nationale de la
Recherche, under grant ANR-09-BLAN-0216-01 (MEGAS).
\end{acknowledgement}

\bibliographystyle{plain}
\bibliography{legoll_lelievre_remarks}

\end{document}